%% file: ms.tex
\documentclass[a4paper,twoside,12pt]{amsart}
\usepackage{amssymb,amsmath}
\usepackage{tikz}
\usepackage{ifpdf}
\ifpdf
  \usepackage[colorlinks,linkcolor=red,anchorcolor=blue,citecolor=green]{hyperref}
\fi

\newtheorem{theorem}{Theorem}[section]

\newtheorem{claim}{Claim}[theorem]

\newtheorem{corollary}[theorem]{Corollary}

\newtheorem{lemma}[theorem]{Lemma}

\newtheorem{proposition}[theorem]{Proposition}
\newtheorem{question}[theorem]{Question}

\theoremstyle{definition}
\newtheorem{definition}[theorem]{Definition}

\theoremstyle{remark}
\newtheorem{remark}[theorem]{Remark}

\newenvironment{poclaim}
 {\begin{proof}[Proof of Claim~\theclaim]
  \expandafter\renewcommand\expandafter\qedsymbol\expandafter
   {\qedsymbol\enspace Claim~\theclaim}}
 {\end{proof}}

\numberwithin{equation}{section}

\newcommand{\RCA}{\mathsf{RCA}_0}
\newcommand{\WKL}{\mathsf{WKL}_0}
\newcommand{\WKLax}{\mathsf{WKL}}
\newcommand{\ACA}{\mathsf{ACA}_0}
\newcommand{\ATR}{\mathsf{ATR}_0}

\newcommand{\WWKL}{\mathsf{WWKL}_0}
\newcommand{\RAN}{\mathsf{RAN}}
\newcommand{\POS}{\mathsf{POS}}
\newcommand{\RT}{\mathsf{RT}}
\newcommand{\SRT}{\mathsf{SRT}}

\newcommand{\TT}{\mathsf{TT}}
\newcommand{\PA}{\mathrm{PA}}
\newcommand{\WPHP}{\mathrm{WPHP}}

\let\leq\leqslant
\let\geq\geqslant
\let\phi\varphi
\renewcommand{\exp}{\mathrm{exp}}
\let\sup\undefined
\let\inf\undefined
\DeclareMathOperator{\sup}{sup}
\DeclareMathOperator{\inf}{inf}
\let\emptyset\varnothing
\let\cross\times
\renewcommand{\times}{\mathbin{\cdot}}

\newcommand{\partref}[1]{\mbox{(\ref{#1})}}

\newcommand{\defeq}{%
 \mathrel{\vcenter{\baselineskip.7ex \lineskiplimit0pt
                   \hbox{\scriptsize.}\hbox{\scriptsize.}}}=%
}
\newcommand{\hyp}{\text{-}}
\newcommand{\defm}[1]{\emph{#1}}
\newcommand{\proves}{\vdash}
\newcommand{\nproves}{\nvdash}

\newcommand{\then}{\rightarrow}
\newcommand{\nsc}{\leftrightarrow}

\newcommand{\compl}{^{\mathrm c}}
\newcommand{\Img}{\mathrm{Im}}
\newcommand{\IN}{\mathbb{N}}
\newcommand{\QQ}{\mathbb{Q}}
\newcommand{\evalin}[1]{^{#1}}
\newcommand{\elemsub}{\preccurlyeq}
\newcommand{\elemext}{\succcurlyeq}
\newcommand{\Def}{\mathrm{Def}}
\newcommand{\Th}{\mathrm{Th}}
\newcommand{\lang}{\mathrm{L}}
\newcommand{\str}[1]{\mathfrak{#1}}
\newcommand{\thoo}[1]{\mathcal{#1}}
\newcommand{\WWKLax}{\mathsf{WWKL}}

\newcommand{\qeq}{\quad=\quad}

\newcommand{\fa}[1]{\forall{#1}\ }
\newcommand{\ex}[1]{\exists{#1}\ }
\newcommand{\exi}[1]{\ex{!#1}\ }
\newcommand{\falt}[2]{\forall{#1{<}{#2}}\ }
\newcommand{\exlt}[2]{\exists{#1{<}{#2}}\ }
\newcommand{\fale}[2]{\forall{#1{\leq}{#2}}\ }
\newcommand{\exle}[2]{\exists{#1{\leq}{#2}}\ }
\newcommand{\fage}[2]{\forall{#1{\geq}#2}\ }
\newcommand{\exge}[2]{\exists{#1{\geq}#2}\ }
\newcommand{\fain}[2]{\forall{#1{\in}{#2}}\ }
\newcommand{\exin}[2]{\exists{#1{\in}{#2}}\ }

\DeclareMathOperator*{\bigvvee}
 {\mathchoice{\bigvee\mkern-16.5mu\bigvee}
             {\bigvee\mkern-12mu\bigvee}
             {\bigvee\mkern-13mu\bigvee}
             {\bigvee\mkern-11mu\bigvee}}

\newcommand{\Tredeq}{\leq_{\mathrm{T}}}
\newcommand{\defd}{\mathclose\downarrow}
\newcommand{\undefd}{\mathclose\uparrow}
\newcommand{\restd}{\mathord{\upharpoonright}}

\newcommand{\Lone}{\lang_1}
\newcommand{\Ltwo}{\lang_2}
\newcommand{\PAminus}{\PA^-}
\newcommand{\ind}{\mathrm{I}}
\newcommand{\bd}{\mathrm{B}}
\newcommand{\IB}{\mathrm{IB}}

\newcommand{\PHP}{\mathrm{PHP}}
\newcommand{\Cd}{\mathrm{C}}

\newcommand{\tuple}[1]{\langle#1\rangle}
\let\seq\tuple
\newcommand{\card}[1]{\mathopen|#1\mathclose|}
\newcommand{\Ackin}[2]{#1\in #2}
\newcommand{\ee}{\mathrm{e}}
\newcommand{\cf}{\mathrm{cf}}
\newcommand{\cfsub}{\subseteq_\cf}
\newcommand{\Tr}{\mathrm{Tr}}
\newcommand{\Con}{\mathrm{Con}}
\newcommand{\Card}{\mathrm{Card}}
\newcommand{\num}{\mathsf{num}}

\newcommand{\floor}[1]{{\left\lfloor#1\right\rfloor}}

\newcommand{\UF}{\mathcal{U}}

\newcommand{\dname}[1]{{\setbox0=\hbox{$#1$}%
 \ooalign{\unhbox0\crcr\hidewidth\lower.5ex\hbox{$.$}\hidewidth}%
}}

\begin{document}

\title{Where Pigeonhole Principles meet K\"{o}nig Lemmas}


\author[Belanger]{David Belanger}
\address{Department of Mathematics: Analysis, Logic and Discrete Mathematics\\
Ghent University}
\email{david.belanger@ugent.be}

\author[Chong]{C.~T.~Chong}
\address{Department of Mathematics\\
National University of Singapore\\Singapore 119076}
\email{chongct@nus.edu.sg}

\author[Wang]{Wei Wang}
\address{Institute of Logic and Cognition and Department of Philosophy\\Sun Yat-Sen University\\Guangzhou, China}
\email{wwang.cn@gmail.com}

\author[Wong]{Tin Lok Wong}
\address{Department of Mathematics\\National University of Singapore\\Singapore 119076}
\email{matwong@nus.edu.sg}

\author[Yang]{Yue Yang}
\address{Department of Mathematics\\National University of Singapore\\Singapore 119076}
\email{matyangy@nus.edu.sg}

\thanks{Chong's research was partially supported by NUS grants C-146-000-042-001 and WBS : R389-000-040-101.
Wang was partially supported by China NSF Grant 11471342.
 Wong was financially supported by
 the Singapore Ministry of Education Academic Research Fund
  Tier~2 grant MOE2016-T2-1-019 / R146-000-234-112
   when this research was carried out.
All the authors acknowledge the support of JSPS--NUS grants
 R146-000-192-133 and R146-000-192-733 during the course of the work.}

\begin{abstract}
We study
 the pigeonhole principle for $\Sigma_2$-definable injections
  with domain  twice as large as the codomain, and
 the weak K\"onig lemma for $\Delta^0_2$-definable trees
  in which every level has at least half of the possible nodes.
We show that the latter implies the existence of $2$-random reals, and
 is conservative over the former.
We also show that the former is strictly weaker than
 the usual pigeonhole principle for $\Sigma_2$-definable injections.
\end{abstract}

\maketitle

\input{introduction.tex}

\input{basics.tex}

\input{wwkl.tex}

\input{fo-theory.tex}

\input{bsigma-wphp-csigma.tex}

\input{further.tex}

\section*{Acknowledgements}
We thank Keita Yokoyama for numerous fruitful discussions
 which led to a simplification of the proof of Theorem~\ref{thm:WPHP-BSigma}
  and to the conception of Proposition~\ref{prop:in=cl}.
We thank Leszek Ko\l odziejczyk for introducing to us
 the references relevant to the $\Delta_0$~pigeonhole principle.
We thank Ali Enayat for bringing to our attention
 Blanck's recent preprint~\cite{unpub:hier-incompl}.

\bibliography{ms}
\bibliographystyle{plain}

\end{document}

%% file: introduction.tex
\section{Introduction}
As Stephen Simpson maintained in his book~\cite{Simpson:2009.SOSOA},
 the goal of reverse mathematics is to investigate
  which set existence axioms are needed to prove  theorems of ordinary mathematics.
Since a great amount of ordinary mathematics can be formalized
  in the framework of second-order arithmetic through the process of coding,
 one may regard the set existence axioms
  to be those concerned with subsets of $\mathbb{N}$.
Our focus in this paper is on  the following slightly different question:
\begin{quote}
  Which elementary or first-order properties of $\mathbb{N}$ are needed
   to prove  theorems of ordinary mathematics?
\end{quote}
We study one instance of this very broad question.

To motivate this line of thought, we briefly review some recent developments in reverse mathematics.
Traditionally the most prominent axiom systems about the existence of subsets of $\mathbb{N}$
 are the so-called \emph{Big Five} systems,
  i.e., $\RCA$, $\WKL$, $\ACA$, $\ATR$ and $\Pi^1_1\hyp\mathsf{CA}_0$.
The weakest system~$\RCA$,
  whose principal constituents are the induction scheme for $\Sigma^0_1$-formulas and
   the comprehension scheme asserting the existence of all $\Delta^0_1$-definable sets,
 is usually taken---as we do in this paper---to be the base system.
Over $\RCA$, many important theorems in ordinary mathematics
 are known to be equivalent to one of the Big Five.
However, in the last two decades a growing body of exceptions have appeared:
 a number of theorems in ordinary mathematics were found to be inequivalent to any of the Big Five, and
 among them some  were proved to be independent of each other.
Classical computability-theoretic methods over the standard natural numbers
 have been popular and fruitful in driving this development.
They provide powerful tools for constructing models of the form $(\mathbb{N}, \thoo{S})$,
  where $\thoo{S}$ is a subset of the power set of $\mathbb{N}$,
 to establish  independence results.

Recently,
 the study of first-order strength of a given subsystem of second-order arithmetic
  has attracted much attention.
In such  studies,
 model-theoretic and proof-theoretic techniques come in naturally.
This approach has provided insights into reverse mathematics
 that the computability-theoretic approach does not.
It has introduced avenues for answering open questions in reverse mathematics,
  and in fact questions  not answerable using standard models $(\mathbb{N},\thoo{S})$, by  nature of the questions themselves.
The best-known examples are all concerned with Ramsey's theorem for pairs~($\RT^2_2$):
 first, the theorem of Chong, Slaman and Yang~\cite{Chong.Slaman.ea:2014.SRT}
  which separates $\RT^2_2$ from its stable version ($\SRT^2_2$);
 second, the theorem by the same authors \cite{Chong.Slaman.ea:2017}
  that  $\RT^2_2$ does not imply
   the induction scheme for $\Sigma^0_2$-formulas; and
 third, the theorem of Patey and Yokoyama~\cite{Patey.Yokoyama:2018}
  which says that all $\Pi^0_3$-consequences of $\RT^2_2$ are already provable in $\RCA$.

A particularly interesting aspect of the first example above is
 that the statements of $\RT^2_2$ and $\SRT^2_2$
  make no direct reference to the first-order properties of $\mathbb{N}$.
Conceivably, this independence result can be reproduced
 using classical computability-theoretic techniques.
On the other hand, there are a number of second-order statements
 whose strengths can only be understood by studying their first-order consequences.
An example is $\TT^1$,
 which is a tree version of the infinite pigeonhole principle
  asserting,  for every partition of~$2^{<\mathbb{N}}$ into finitely many parts,
   the existence of a monochromatic subtree isomorphic
   to the (infinite) perfect binary tree~$2^{<\mathbb{N}}$.
   Any standard model of $\RCA$
  is trivially  a model of~$\TT^1$,
 because in this case a monochromatic subtree can be computed from the partition.
Indeed, $\RCA + \ind\Sigma^0_2 \vdash \TT^1$.
However, Corduan et al.~\cite{Corduan.Groszek.ea:2010.TT1} proved that
 in the absence of $\ind\Sigma^0_2$
  a finite partition of $2^{<\mathbb{N}}$ may fail to compute
   a monochromatic subtree isomorphic to $2^{<\mathbb{N}}$.
As $\RCA + \TT^1 \vdash \bd\Sigma^0_2$ easily,
 the strength of $\TT^1$ lies between $\bd\Sigma^0_2$ and $\ind\Sigma^0_2$, and
  thus cannot be calibrated in the classical computability-theoretic setting.
Recently, Chong, Li, Wang and Yang~\cite{Chong.Li.ea:2019} proved that $\RCA + \TT^1 \nproves \ind\Sigma^0_2$.

\begin{figure}[tp]\centering
\begin{tikzpicture}[y=-1cm]
\node(2wwkl0)     at (0  ,0  ) {$2\hyp\WWKL$};
\node(bs02)       at (5.7,1  ) {$\bd\Sigma^0_2$};
\node(2wwkl0half) at (0  ,3.5) {$2\hyp\WWKL(1/2)$};
\node(s02wphp)    at (5.7,3.5) {$\Sigma^0_2\hyp\WPHP$};
\node(2ran)       at (0  ,5.5) {$2\hyp\RAN$};
\draw[->] (2wwkl0)    --node[right]{obvious}(2wwkl0half);
\draw[->] (2wwkl0half)--node[right]{Corollary~\ref{cor:2-WWKL-2-RAN}}(2ran);
\draw[->] (2wwkl0half)--node[above]{$\Pi^1_1$-conservative}
                        node[below]{Theorem~\ref{thm:FOT_2-WWKL-half}}(s02wphp);
\draw[->] (2wwkl0)    --node[above right]{Avigad, Dean, Rute~\cite{Avigad.Dean.ea:2012}}(bs02);
\draw[->] (s02wphp)   ..controls+(-125:1.3)and+(125:1.3)..
                        node{|} node[pos=.8,left]{Theorem~\ref{thm:WPHP-BSigma}}(bs02);
\draw[->] (bs02)      ..controls+(55:1.3)and+(-55:1.3)..
                        node[right]{\begin{tabular}{@{\!}l@{\!}}
                         Dimitracopoulos,\\
                         Paris~\cite{Dimitracopoulos.Paris:1986}
                        \end{tabular}}(s02wphp);
\draw[->] (2ran)      ..controls+(-180:2.3)and+(-180:2.3)..
                        node{|} node[pos=.55,right]{Slaman} (2wwkl0);
\end{tikzpicture}
\caption{Some relationships between various subsystems of second-order arithmetic
 investigated in this paper, over~$\RCA$}
\end{figure}
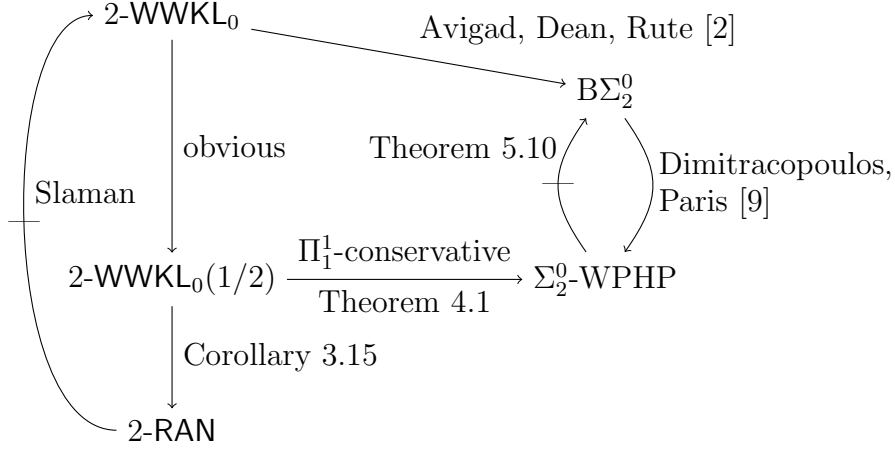
Another example is $2\hyp\WWKL(1/2)$, which we investigate in this paper.
It is related to measure theory and algorithmic randomness.
Recall the principle $\WWKL$, introduced by Yu and Simpson~\cite{Yu.Simpson:1990}, which states that if a binary tree $T$ satisfies
\begin{equation}\label{eq:WWKL}
  \exists m > 0 \forall n (|T \cap 2^n| > 2^n/m),
\end{equation}
where $2^n$ denotes both a natural number (on the right-hand side of the inequality) and the set of all binary strings of length $n$ (on the left-hand side) then there exists an $X \in [T]$, meaning $X$ is an infinite path through $T$.
This principle is known to be strictly weaker than $\WKL$, but independent of $\RCA$.
Avigad et al.\ considered in a more recent paper~\cite{Avigad.Dean.ea:2012}
 the analogue of~$\WWKL$ for $\Delta^0_n$-definable trees,
  which they call $n\hyp\WWKL$.
They showed that $2\hyp\WWKL$ is equivalent to
  a formalized version of the dominated convergence theorem in second-order arithmetic~($\mathsf{DCT}$), and
 $n\hyp\WWKL$ implies the existence of $n$-random reals~($n\hyp\mathsf{RAN}$).
In the same paper, Avigad et al.\ asked whether $n\hyp\RAN$ is equivalent to $n\hyp\WWKL$.
Slaman~[unpublished] answered their question in the negative.
By relativizing an argument in Ku\v{c}era~\cite[Lemma 3]{Kucera:85},
 one can show that, over the standard model $\mathbb{N}$,
  every $n$-random real computes some $X \in [T]$,
   whenever $T$ is a $\Delta^0_n$-definable tree satisfying~\eqref{eq:WWKL}.
Hence the use of classical computability-theoretic tools alone cannot answer the question of Avigad et al.

In this paper we improve on Slaman's result.
We introduce a principle called~$2\hyp\WWKL(1/2)$
 whose strength lies between those of $2\hyp\WWKL$ and~$2\hyp\RAN$.
It states that if $T$ is a $\Delta^0_2$-definable tree satisfying
\begin{equation}\label{eq:2-WWKL-half}
  \forall n (|T \cap 2^n| > 2^{n-1})
\end{equation}
then there exists an $X \in [T]$.
We prove that the first-order theory of $2\hyp\WWKL(1/2)$
 can be axiomatized by $\ind\Sigma_1$ plus the principle $\Sigma_2\hyp\WPHP$,
  which is a variant of the finite pigeonhole principle
   strictly weaker than $\bd\Sigma_2$.
As $2\hyp\WWKL$ implies $\bd\Sigma^0_2$~\cite[Theorem 3.7]{Avigad.Dean.ea:2012},
 we know $2\hyp\WWKL(1/2)$ is strictly weaker than $2\hyp\WWKL$.
From this we conclude that $\ind\Sigma_1+\Sigma_2\hyp\WPHP$ is an upper bound
  for the first-order theory of $2\hyp\RAN$.
We also prove that $\Sigma_2\hyp\WPHP$ is substantially different
 from the usual fragments of first-order arithmetic.

We organize this paper as follows.
In Section~\ref{s:basics},
 we introduce some basic notation and
    set up a few preliminary results about weak pigeonhole principles.
In Section~\ref{s:wwkl},
 we explore the notion of $\Delta^0_2$ trees in the absence of $\bd\Sigma^0_2$.
In particular,
 we define $2\hyp\WWKL(1/2)$, and
  verify that
   $\RCA + 2\hyp\WWKL(1/2)\proves\Sigma^0_2\hyp\WPHP\wedge2\hyp\RAN$ there.
In Section~\ref{s:fo-thy},
 we prove that $2\hyp\WWKL(1/2)$ is $\Pi^1_1$-conservative over $\RCA + \Sigma^0_2\hyp\WPHP$.
In Section~\ref{s:wph-fragments},
 we prove that $\Sigma_{n+1}\hyp\WPHP$ is strictly weaker than $\bd\Sigma_{n+1}$,
  but strictly stronger than the cardinality scheme for $\Sigma_{n+1}$ formulas ($\Cd\Sigma_{n+1}$).
We conclude in Section~\ref{s:further} with
 a discussion of the techniques developed
  in our study of pigeonhole principles, and
 a list of questions.

%% file: basics.tex
\section{Basics}\label{s:basics}

Let us start with some notational conventions.
The language ~\defm{$\Lone$} of first-order arithmetic
 has symbols $0,1,{+},{\times},{<}$ and a symbol for equality.
By convention,
 if $\str M,\str N,\str N_k,\dots$ are $\Lone$~structures,
  then their universes are always denoted by $M,N,N_k,\dots$ respectively.
The language ~\defm{$\Ltwo$} of second-order arithmetic
 has a first-order sort and a second-order sort,
  with a copy of~$\Lone$ on the first-order sort, and
   a symbol~$\in$ relating a first-order object to a second-order object.
We only consider $\Ltwo$~structures
 \begin{math}
  \str M=(M,\thoo S,0\evalin{\str M},1\evalin{\str M},
          {+}\evalin{\str M},{\times}\evalin{\str M},
          {<}\evalin{\str M},{\in}\evalin{\str M})
 \end{math}
 where the second-order universe~$\thoo S$
   is a collection of subsets of the first-order universe~$M$, and
  ${\in}\evalin{\str M}={\in}$.
Since there is no risk of ambiguity in this paper,
 we abbreviate an $\Ltwo$~structure $(M,\thoo S,\dots)$ as $(M,\thoo S)$.

Recall that $\PA^-$ is a finite set of axioms
 saying that the model is the non-negative half of a discretely ordered commutative ring.
The \defm{induction axiom} $\ind\varphi$ for a formula $\varphi(x,\vec{Y})$,
  where $\vec{Y}$ is a finite tuple of first- or second-order parameters,
 is the sentence
\[
  \forall \vec{Y} (\varphi(0,\vec{Y}) \wedge \forall x(\varphi(x,\vec{Y}) \to \varphi(x+1,\vec{Y})) \to \forall x \varphi(x,\vec{Y})).
\]
If $\Gamma$ is a set of formulas,
 then $\ind\Gamma$ denotes the \defm{induction scheme} for $\Gamma$,
  i.e.,~the collection of all $\ind\varphi$'s in which $\varphi \in \Gamma$.
The \defm{bounding axiom} $\bd\psi$ for a formula $\psi(w,x,\vec{Y})$ states
\[
  \forall \vec{Y}, u(\forall w < u \exists x \psi(w,x,\vec{Y}) \to \exists v \forall w < u \exists x < v \psi(w,x,\vec{Y})).
\]
If $\Gamma$ is a set of formulas,
 then $\bd\Gamma$ denotes the \defm{bounding scheme} for $\Gamma$,
 i.e.,~the collection of all $\bd\psi$'s such that $\psi \in \Gamma$.

Fix a $\Pi_2$~sentence $\exp$ which asserts the totality of exponentiation
 over~$\ind\Delta_0$.
One can expand a model $\str M\models\ind\Delta_0+\exp$
 with the function $x\mapsto2^x$ provided by~$\exp$.
We sometimes identify $\str M$ with this expansion.
From Gaifman--Dimitracopoulos~\cite[Theorem~3.3]{incoll:fragPA+MRDP},
 we know $\str M\models\ind\Delta_0(\exp)$,
  where \defm{$\Delta_0(\exp)$} denotes the smallest collection of formulas
   that contains all the atomic formulas \emph{in the expanded language}, and
    is closed under Boolean operations and bounded quantification.
We may even allow the exponential function to appear
 in the bounding terms here~\cite[Lemma~I.1.30]{book:hajek+pudlak}.

We say a subset $F\subseteq M$ is \defm{$\str M$-finite} or \defm{coded in $\str M$}
 if there is $c\in M$ such that
  \begin{equation*}
   F=\{x\in M:\str M\models\text{the $x$-th digit in the binary expansion of $c$ is $1$}\}.
  \end{equation*}
For example, the finite set $F=\{0,1,3\}$ is $\str M$-finite
  in all models $\str M\models\ind\Delta_0+\exp$,
 with code $c=2^0+2^1+2^3=11$.
In a model $\str M\models\ind\Delta_0+\exp$ many desirable properties of $\str M$-finite sets are available,
 for example, bounded $\Delta_0$~comprehension~\cite[Theorem~I.1.36]{book:hajek+pudlak},
  the pigeonhole principle for coded functions~\cite[Theorem~I.1.41(2)]{book:hajek+pudlak}, and
  the existence of cardinalities~\cite[Theorem~I.1.41(1)]{book:hajek+pudlak}.
Unless otherwise stated,
 we denote by $\card X$ the cardinality of a set~$X$ in this sense.
Recall that $\ind\Delta_0$ is already enough to prove
 the usual properties of the Cantor pairing function $(x,y)\mapsto\tuple{x,y}$.
Similarly, many properties of
 the usual sequence-coding function $(a_i:i<n)\mapsto\seq{a_i:i<n}$
  can be proved in $\ind\Delta_0+\exp$.
Note that both of these coding functions have $\Delta_0$~graphs, and
 each component of the object being coded
  is bounded above by the code, provably in $\ind\Delta_0$.

The axiom system $\RCA$ consists of $\PA^-$, $\ind\Sigma^0_1$ and
\begin{equation}\label{eq:RCA}
  \forall\vec Y\bigl(\forall x(\varphi(x,\vec Y) \leftrightarrow \psi(x,\vec Y))
   \to \exists Z \forall x(x \in Z \leftrightarrow \varphi(x,\vec Y))
  \bigr)
\end{equation}
for each pair $\varphi$ and $\psi$ of $\Sigma^0_1$ and $\Pi^0_1$ formulas.
From a computability-theoretic viewpoint,
 the scheme~\eqref{eq:RCA} tells us that
  if $\vec Y$ is a tuple of parameters from a model $(M, \thoo{S})\models\RCA$, and
    $\bigoplus\vec Y$ pointwise computes a total, binary-valued function $f\colon M \to 2$,
   then $f^{-1}(1) \in \thoo{S}$.

Since we will sometimes work with sets $Z$ outside of $\thoo S$,
 we must face a troublesome detail documented by Groszek and Slaman~\cite{Groszek.Slaman:1994}:
  if $\bd\Sigma^0_1$ does not hold relative to $Z$,
   then the `pointwise Turing reduction' described in line~(\ref{eq:RCA}) is not transitive,
    i.e.,~there are $X$ and $Y$
     such that $X$ is pointwise computable from $Y$ and $Y$ from $Z$,
      but $X$ is not pointwise computable from $Z$---in symbols,
     $X \in \Delta^{0,Y}_1$ and $Y \in \Delta^{0,Z}_1$, but $X \not \in \Delta^{0,Z}_1$.
For this reason we also introduce a stronger, `setwise' notion of Turing reduction:
 we write $X \Tredeq Y$ if
  the sequence $(X \upharpoonright n : n \in M)$, viewed as a function from $M$ to codes of finite sets,
   is pointwise computable from $Y$;
 here $X \upharpoonright n$ denotes the binary string representing the first $n$ bits of $X$.
In the language of \cite{Groszek.Slaman:1994}, such an $X$ is \emph{strongly recursive in $Y$}.

The classical Shoenfield Limit Lemma states that
 every total $\Delta^0_2$~function $F$
  can be approximated by a recursive total function~$F_0$
 in the sense that $F=\lim_s F_0(\cdot,s)$.
A formalization of it will be used all over this paper.
We formulate the Limit Lemma in the following very general form
 because we want to exploit this extra generality in Section~\ref{s:wph-fragments}.

\begin{lemma}[Limit Lemma]\label{limit-lemma}
Let $n\in\IN$.
For every $\theta(\bar x,y)\in\Sigma^0_{n+1}$,
 there exist $\eta(s,\bar x,y)\in\Sigma^0_n$
         and $\alpha(s,\bar x)\in\Pi^0_n$
  such that $\ind\Sigma^0_n$ proves
  \begin{align*}
   &\fa{\bar x}{\exi y{\theta(\bar x,y)}}\\
   &\then\fa{s,\bar x}{\exi y{\eta(s,\bar x,y)}}\\
   &\wedge\fa{\bar x,y}{\bigl(
     \theta(\bar x,y)
     \nsc\ex s{\fage ts{\eta(t,\bar x,y)}}
    \bigr)}\\
   &\wedge\fa{s,\bar x}{\bigl(
     \alpha(s,\bar x)
     \nsc\fage ts{\fa{y,y'}{\bigl(
      \eta(s,\bar x,y)
      \wedge\eta(t,\bar x,y')
      \then y=y'
     \bigr)}}
    \bigr)}.
  \end{align*}
Here the formula $\theta(\bar x,y)$ may contain undisplayed free variables,
 in which case we allow the same variables
  to appear free in $\eta(s,\bar x,y)$ and $\alpha(s,\bar x)$.
\end{lemma}

\begin{proof}
For the $n\geq1$ case, see Theorem~I.3.2 in H\'ajek--Pudl\'ak~\cite{book:hajek+pudlak}.
Note that $\bd\Sigma^0_n$ is sufficient for this case.
If $n=0$ and $\theta=\ex u{\theta_0}$ where $\theta_0\in\Sigma^0_0$,
 then we can define $\eta(s,\bar x,y)$ and $\alpha(s,\bar x)$ to be
  \begin{equation*}
   \bigl(y\leq s\wedge\exle us{\theta_0(u,\bar x,y)}\bigr)
   \vee\bigl(\fale{y',u}s{\neg\theta_0(u,\bar x,y')}\wedge y=s\bigr)
  \end{equation*}
  and
  \begin{math}
   \exle{y,u}s{\theta_0(u,\bar x,y)}
  \end{math} respectively.
\end{proof}

To formulate our pigeonhole principles,
 let us borrow the Erd\H{o}s--Rado arrow notation
\[
  \kappa \to (\lambda)^n_c,
\]
 which means that every $c$-coloring of the $n$-element subsets of $\kappa$
  admits a homogeneous subset of cardinality $\lambda$.

\begin{definition}
When $\Gamma$ is a class of functions,
\[
  \Gamma\colon x \to (z)^1_y
\]
 means that if $f\colon x\to y$ is in~$\Gamma$
 then $f$ is constant on a subset of $x$ of cardinality $z$.
Here every number~$x$ is identified with $\{v:v<x\}$ as in set theory.
\end{definition}

For example, in this notation,
 the usual $\Sigma^0_{n+1}$ Pigeonhole Principle
  can be written as $\forall x(\Sigma^0_{n+1}\colon x+1\to(2)^1_x)$.

We first prove several first-order properties of these partition relations.
The first-order results we present here can easily be relativized to the second-order setting.

As is well known,
 there is a level-by-level correspondence
  between the usual Pigeonhole Principle and the collection scheme.
\begin{theorem}[Dimitracopoulos and Paris \cite{Dimitracopoulos.Paris:1986}]\label{thm:B=PHP}
Over $\ind\Delta_0+\exp$,
 \begin{equation*}
  \forall x(\Sigma_{n+1}\colon x+1 \to (2)^1_x)
 \end{equation*} is equivalent to $\bd\Sigma_{n+1}$
 for all $n\in\IN$. \qed
\end{theorem}

The next lemma, which can be viewed as a variant of Theorem~\ref{thm:B=PHP},
 offers a number of $\bd\Sigma_{n+1}$-like characterizations of the property
  $\Sigma_{n+1}\colon b \rightarrow (2)^1_a$.
\begin{lemma}\label{lem:WPH-vs-convergence}
Fix $n\in\IN$.
Suppose that $\str M \models \ind\Sigma_n+\exp$ and $a < b$ are elements of $M$.
The following are equivalent:
\begin{enumerate}
  \item \label{lem:WPHP-vs-convergence/1}
    $\str M\models\Sigma_{n+1}\colon b \to (2)^1_{a}$.
    \item \label{lem:WPHP-vs-convergence/2}
     If $\phi$ is a $\Pi_{n}$ formula and $\str M \models (\forall x < b) (\exists y) \varphi(x,y)$, then there exist an $\str M$-finite $A \subseteq b$ of size $|A| \geq a$ and an $\ell \in M$ such that
        $$
            \str M \models (\forall x \in A)(\exists y < \ell) \varphi(x,y).
        $$
    \item \label{lem:WPHP-vs-convergence/3}
     As above, but $\phi$ is $\Sigma_{n+1}$.
    \item \label{lem:WPHP-vs-convergence/4}
     For every $\Delta_{n+1}$ function $f\colon b \to M$ there exist an $\str M$-finite $A \subseteq b$ of size $|A| \geq a$ and $\ell\in M$ such that $f(x) < \ell$ for all $x\in A$.
    \item \label{lem:WPHP-vs-convergence/5}

 As above, plus $\{(x,f(x)) : x \in A\}$ is $\str M$-finite.
\end{enumerate}
\end{lemma}

\begin{proof}
$(\ref{lem:WPHP-vs-convergence/1} \implies \ref{lem:WPHP-vs-convergence/2})$.
Define $m\colon b\to M$ by setting
 \begin{equation*}
  \text{$m(x) = \tuple{x,y}$ whenever $y$ is least s.t.\ $\str M \models\phi(x,y)$}.
 \end{equation*}
For each $x<b$,
 the set $\{v<b:m(v)<m(x)\}$ is $\str M$-finite by $\ind\Sigma_n+\exp$, and
 so it has a cardinality, say $f(x)$, from the point of view of~$\str M$.
This gives a function $f\colon b\to M$
 which is injective (since $m$ is injective) and $\Sigma_{n+1}$.
Hence $\Img(f)\not\subseteq a$ by~\partref{lem:WPHP-vs-convergence/1}.
If $x<b$ such that $f(x)=\card{\{v<b:m(v)<m(x)\}}\geq a$,
 then we can set $A=\{v<b:m(v)<m(x)\}$ and $\ell=m(x)$.

$(\ref{lem:WPHP-vs-convergence/2} \implies \ref{lem:WPHP-vs-convergence/3})$.
Write $\phi$ as $(\exists z)\psi(x,y,z)$, and apply $(\ref{lem:WPHP-vs-convergence/2})$ to $\psi$.

$(\ref{lem:WPHP-vs-convergence/3} \implies \ref{lem:WPHP-vs-convergence/4})$.
Suppose $f(x)=y$ is defined by the $\Sigma_{n+1}$~formula $\phi(x,y)$ in~$\str M$.
Then $M \models (\forall x < b)(\exists y) \phi(x,y)$,
 so $(\ref{lem:WPHP-vs-convergence/3})$ gives the required $A$ and~$\ell$.

$(\ref{lem:WPHP-vs-convergence/4} \implies \ref{lem:WPHP-vs-convergence/5})$.
Suppose $f(x) = y$ is defined by $(\exists z)\psi(x,y,z)$, where $\psi$ is $\Pi_n$.
Define
 \begin{equation*}
  g(x)=\min\{\tuple{y,z}\in M:\str M \models\psi(x,y,z)\}.
 \end{equation*}
Then $g$ is $\Delta_{n+1}$.
Apply $(\ref{lem:WPHP-vs-convergence/4})$ to get $A\subseteq b$ of size $\geq a$ and $\ell\in M$
 such that $x \in A$ implies $g(x) < \ell$.
Now the set in \partref{lem:WPHP-vs-convergence/5} is equal to
 \begin{equation*}
  \{(x,y) : \text{$x \in A$ and $\str M\models(\exists z < \ell)\psi(x,y,z)$}\},
 \end{equation*}
  which is $\Pi_n$ and bounded, and is therefore $\str M$-finite by $\ind\Sigma_n+\exp$.

$(\ref{lem:WPHP-vs-convergence/5} \implies \ref{lem:WPHP-vs-convergence/1})$.
Suppose $f\colon b \rightarrow a$ is a $\Delta_{n+1}$ injection.
Apply $(\ref{lem:WPHP-vs-convergence/5})$ to get an $A \subseteq b$ of size $\geq a$
 on which the graph $\{(x,y):\text{$x \in A$ and $f(x)=y$}\}$ of~$f$ is $\str M$-finite.
Through this and the injectivity of $f$, it follows that the image $f(A)$ also has size $\geq a$, so that $f(A)$ must equal the whole codomain $a$. But this is impossible, as $f$ is an injection and $A$ is a proper subset of the domain $b$. (As a closing sidenote: this is the only place where we use $a < b$.)
\end{proof}

Of particular interest to this paper is the pigeonhole principle
 for injections with domain twice as large as the codomain.

\begin{definition}[$\Gamma$ Weak Pigeonhole Principle]\label{dfn:WPH}
Given a class $\Gamma$ of functions, $\Gamma\hyp\WPHP$ is the statement
\[
  \forall x\geq1(\Gamma\colon 2x \to (2)^1_x),
\]
i.e.,~for no positive~$x$ is there an injection in~$\Gamma$
 from $2x$ to~$x$.
\end{definition}

It is natural to ask how the weak pigeonhole principle relates
 to other, similar statements about definable injections.
We start by observing:
 \begin{lemma}\label{lem:WPH-newlemma}
  Let $n\in\IN$ and $\str M\models\ind\Sigma_n+\exp$.
  Suppose we have a $\Sigma_{n+1}$-definable injection
   $f\colon b\to a$ in~$\str M$, where $a<b$.
  Then:
  \begin{enumerate}
   \item For each $d \in M$ there is a $\Sigma_{n+1}$ injection mapping $d b \rightarrow d a$.\label{part:WPH-newlemma/1}
   \item For each $c < a$ in $M$, either there is a $\Sigma_{n+1}$ injection mapping $b \rightarrow c$, or there is one mapping $(b - c) \rightarrow (a - c)$.\label{part:WPH-newlemma/2}
   \item For each nonzero, standard $m \in \mathbb{N}$, there is a $\Sigma_{n+1}$ injection mapping $\lceil b/m\rceil \rightarrow \lfloor a/m\rfloor$.\label{part:WPH-newlemma/3}
  \end{enumerate}
 \end{lemma}
 \begin{proof}
  For the first part, define $g\colon d  b \rightarrow d a$ by $g(q b + r) = q a+f(r)$ whenever $0 \leq r < b$.
  For the second part, assume that there is no $\Sigma_{n+1}$ injection mapping $b \rightarrow c$, i.e.,~that $\str M\models\Sigma_{n+1}\colon b \rightarrow (2)^1_c$. Then by Lemma~\ref{lem:WPH-vs-convergence}$(\ref{lem:WPHP-vs-convergence/5})$ there is an $\str M$-finite $C \subseteq b$ of size $c$ for which $f(C)$ is $\str M$-finite and also of size $c$. So deleting $C$ from the domain and $f(C)$ from the codomain gives us an injection mapping a set of size $b-c$ into one of size $a-c$.
For the third part, by~\partref{part:WPH-newlemma/2},
 either there is a $\Sigma_{n+1}$ injection mapping $b\to\floor{a/m}$,
     or there is one mapping $(b-\floor{a/m})\to(a-\floor{a/m})$.
If it is the former, we are already done;
 if it is the latter, apply \partref{part:WPH-newlemma/2} again
  to get either an injection $(b-\floor{a/m})\to\floor{a/m}$,
             or an injection $(b-2\floor{a/m})\to(a-2\floor{a/m})$.
As the reader can readily check, continuing in this way for up to $m$ steps,
 at some point we produce the required injection.
(This is where we use that $m$ is standard:
 nonstandardly many iterations of this sort would require stronger axioms
  in general.)
 \end{proof}

This has interesting consequences when the three parts work in concert. For example, given an injection mapping $b \rightarrow b/2$, if we let $c = b/4$, then the lemma's second part can provide (in either outcome) an injection mapping $(3/4)b \rightarrow b/4$. Then the lemma's first and third parts together yield an injection $b \rightarrow b/3$---domain the same as we started with, but codomain markedly smaller. By carefully extending this reasoning, one can prove both parts of the following.

\begin{lemma}\label{lem:WPH-basics}
Suppose $\str M \models \ind\Sigma_n+\exp$ where $n\in\IN$, and $b \in M$.
\begin{enumerate}
 \item The following are equivalent.\label{lem:WPH-basics/1}
  \begin{enumerate}
      \item $\str M\models\Sigma_{n+1}\colon 2b \to (2)^1_{b}$.
      \item $\str M\models\Sigma_{n+1}\colon b \to (2)^1_{rb}$ for some $r \in \mathbb{Q}$ strictly between $0$ and $1$.
      \item $\str M\models\Sigma_{n+1}\colon b \to (2)^1_{rb}$ for all $r \in \mathbb{Q}$ strictly between $0$ and~$1$.
  \end{enumerate}
  \item If $\str M\models\Sigma_{n+1}\colon 2b \to (2)^1_{b}$ then $\str M\models\Sigma_{n+1}\colon 2a \to (2)^1_{a}$\label{lem:WPH-basics/2} for all $a < b$. \qed
\end{enumerate}
\end{lemma}

%% file: wwkl.tex
\section{The Weak Weak K\"onig Lemma}\label{s:wwkl}
We begin with a formalization of the principle $2\hyp\WWKL(1/2)$.
In a model $\str{M} = (M, \mathcal{S})$ of $\RCA$,
 a \emph{(binary) tree} is a set $T$ of $\str{M}$-finite binary sequences
  such that every initial segment of an element of $T$ is also an element of $T$.
If $T$ is a tree, then $[T]$ denotes the collection of ($\str{M}$-)infinite binary sequences
 all of whose $\str{M}$-finite initial segments are in $T$.
It is possible that $T \not\in \mathcal{S}$,
 or $[T] \not\subseteq \mathcal{S}$, or
 even $[T] \cap \mathcal{S}=\emptyset$.
If $M = \mathbb{N}$ and $T$ is a tree, then obviously $T \cap 2^n$ is $\str{M}$-finite for each $n \in M$,
 and the Lebesgue measure of $[T]$ is the limit of $|T \cap 2^n|/2^n$ as $n$ tends to infinity.
However, when $M\not=\IN$, the $T \cap 2^n$'s may not be $\str M$-finite.
Even if all the $T \cap 2^n$'s are $\str M$-finite,
 $\lim_n |T \cap 2^n|/2^n$ may not be as reasonable as in $\IN$.
So we have to impose some additional conditions on $T$ when we want to talk about the measure of $[T]$.
There are in fact several sets of such additional conditions,
 but as we shall see, the corresponding restrictions of $2\hyp\WWKL$ are equivalent to one another.

\begin{definition}\label{dfn:2-WWKL-half}
Let $T$ be a tree in a model $\str M \models \RCA$ and $r \in \QQ^{\str M}$.
We say that $\mu([T]) \geq r$
 if for each $n\in M$,
  there exists an $\str M$-finite $S \subseteq T \cap 2^n$
   such that $|S| \geq r 2^n$.

Define $2\hyp\WWKLax(x)$ to be a formula which expresses, over $\RCA$, that $x$ is a rational number (possibly nonstandard), and that
 \begin{equation*}
  \text{if $T$ is a $\Delta^0_2$ tree with $\mu([T])\geq x$
   then $[T]\neq\emptyset$.}
 \end{equation*}
For a positive $r\in\QQ$,
 let $2\hyp\WWKL(r)=\RCA+2\hyp\WWKLax(r)$.
\end{definition}

If $T$ is a $\Delta^0_2$ tree then $[T]$ is a $\Pi^0_2$ class.
So the principal axiom of $2\hyp\WWKL$ is an instance of the axiom $2\hyp\POS$,
 which states that every $\Pi^0_2$ class with positive measure is nonempty.
Avigad et al.~\cite{Avigad.Dean.ea:2012} proved
 that $2\hyp\WWKL$ is equivalent over $\RCA$ to $2\hyp\POS$, and
 that $2\hyp\WWKL$ implies $2\hyp\RAN$.
We will establish parallel results here.
Let us start with the definition of $\Pi^0_2$ classes within $\RCA$.

\begin{definition}
Fix $\str M=(M,\thoo S)\models\RCA$.
A \defm{$\Pi^0_2$ class} in $\str M$ is
 a subset $\thoo A$ of the power set of $M$
   (which we identify with the set $2^M$ of all functions $M\to2$)
  that can be written in the form
   $$\thoo A = \bigcap_{i \in M}\bigcup_{j \in M} \thoo B_{i,j},$$
 where each $\thoo B_{i,j}$ is a basic open set
  (meaning
   \begin{math}
   \thoo B_{i,j}
   = [\sigma_{i,j}]
   = \{X\in2^M :
       \text{$\sigma_{i,j}$ is an initial segment of $X$}
     \}
   \end{math}
   for some $\str M$-finite binary sequence $\sigma_{i,j}$), and
 where the function mapping $(i,j)$ to the code of $\sigma_{i,j}$ is in $\thoo S$.
We say $\thoo A$ is \emph{strictly presented}
 if $i \leq i'$ implies $\bigcup_{j < j'} \thoo B_{i,j} \supseteq \bigcup_{j < j'} \thoo B_{i',j}$
  for every $j'\in M$.
When $\thoo A$ is strictly presented, $\delta \in \QQ^{\mathfrak{M}}$, and $\mu(\bigcup_j \thoo B_{i,j}) \geq \delta$ for all $i\in M$,
 we write $\hat{\mu}(\thoo A) \geq \delta$.
\end{definition}
It is routine to check that every $\Pi^0_2$ class in a model of $\RCA$
  (except perhaps the empty class) can be written in strictly presented form, and
 hence can have its measure compared with rationals in this way;
 let us stress, however, that without $\bd \Sigma^0_2$,
  this measure may depend partially on the choice of the `presentation'.

\begin{definition}\label{dfn:2-POS-half}
Let $2\hyp\POS(x)$ be a formula which expresses, over $\RCA$, that $x$ is a rational number (possibly nonstandard), and that
  \begin{equation*}
   \text{if $\thoo A$ is a $\Pi^0_2$ class and $\hat{\mu}(\thoo A)\geq x$
    then $\thoo A\neq\emptyset$.}
  \end{equation*}
\end{definition}

Recall that if $T$ is a $\Delta^0_2$ tree then $[T]$ is a $\Pi^0_2$ class.
So we have two kinds of measure inequalities: $\mu([T]) \geq r$ and $\hat{\mu}([T]) \geq r$.
In Lemma \ref{lem:POS-implies-WWKL} below we will show that $\mu([T]) \geq r $ implies $\hat{\mu}([T]) \geq r$ over $\RCA$,
 and hence $\RCA \vdash 2\hyp\POS(r) \to 2\hyp\WWKL(r)$ for all positive $r \in \mathbb{Q}$; and later, in Theorem \ref{thm:WWKL=POS}, we obtain a strengthening and a converse.  But first, let us pause for a useful lemma about $\Delta^0_2$ trees.

One way to state the $n=1$ case of the Limit Lemma \ref{limit-lemma} is: Given a $\Delta^0_2$ set $A$ of natural numbers, there is a sequence of sets $\langle A_0, A_1,\ldots\rangle$ which converges pointwise to $A$ and is uniformly $\Delta^0_1$, i.e. the function $F(x,s)$ which $= 1$ if $x \in A_s$ and $=0$ if $x \not \in A_s$, is in the second-order part of the model. The following refinement of the Limit Lemma says that if, in addition, $A$ is (the set of codes of strings in) a binary tree, then the approximating sets can also be trees (i.e.\ can be closed under initial segment).

\begin{lemma}[$\RCA$] \label{lem:mono-apprx-tree}
 Every $\Delta^0_2$ tree $T \subseteq 2^{<M}$ is the limit of a uniformly $\Delta^0_1$ sequence of trees $\langle T_0, T_1, \ldots \rangle$.
\end{lemma}
With access to $\bd\Sigma^0_2$, this lemma would be immediate. Since only $\RCA$ is available, however, we resort to a `tame cuts'-style contruction.
\begin{proof}
Let $\langle A_0, A_1, \ldots\rangle$ be a uniformly $\Delta^0_1$ sequence of sets approximating $T$ as given by the Limit Lemma, and let
  $$m(\sigma,s) =\min\left\{r \geq |\sigma| : \forall t \left(r \leq t \leq s \implies [\sigma \in A_t \iff \sigma \in A_s]\right)\right\}.$$
In other words, $m$ is a stage-by-stage approximation to the usual modulus function. For each $s$, define a set $T_s$ of strings by:
 $$\sigma \in T_s \mathrm{~if~}(\forall \tau \subseteq \sigma)(\exists \rho \supseteq \tau) \left[\rho \in A_s \mathrm{~and~}m(\rho,s) \leq m(\tau,s) \right].$$
It is immediate from this definition that each $T_s$ is a tree. And since $m(\rho,s)$ is by definition always $\geq |\rho|$, to determine whether a given $\sigma$ is in $T_s$, we need only consider strings $\rho$ and $\tau$ of length $\leq s$; hence the trees are uniformly $\Delta^0_1$. It remains only to verify that they converge pointwise to $T$. For a given $\sigma \in 2^{<M}$, there are two cases to consider: either $\sigma$ is in $T$, or it is not.

 If $\sigma \in T$: Let $r \geq |\sigma|$ be least such that $t \geq r$ implies $\sigma \in A_t$. Since $T$ is downward closed, there is by $\ind \Sigma_1$ an $s \geq r$ such that for all $\tau \subseteq \sigma$ either $\tau \in A_{s}$ or $m(\tau, s) > r$. In particular, for any $t \geq s$ and any $\tau \subseteq \sigma$, we have either $\tau \in A_t$ or $m(\tau,t) > r$. Since $r = m(\sigma,t)$, this means $\sigma$ is in $T_t$ for all $t \geq s$.

 If $\sigma \not \in T$: Let $r \geq |\sigma|$ be least such that $t \geq r$ implies $\sigma \not \in A_t$. By $\bd\Sigma_1$ (and using the fact that $m(\tau,t)$ is always at least $|\tau|$) there is an $s \geq r$ such that for all $\tau \supseteq \sigma$, either $\tau \not \in A_{s}$ or $m(\tau,s) > r$. And in particular, for any $t \geq s$ and any $\tau \supseteq \sigma$, we have either $\tau \not \in A_t$ or $m(\tau,t) > r$. Since $r= m(\sigma,t)$, this means $\sigma$ is not in $T_t$ for any $t \geq s$.
\end{proof}
We put this lemma straight to work.
\begin{lemma}[$\RCA$]\label{lem:POS-implies-WWKL}
If $r \in \mathbb{Q}$ is positive and $T$ is a $\Delta^0_2$ tree with $\mu([T]) \geq r$ according to Definition \ref{dfn:2-WWKL-half}, then $[T]$ can be strictly presented as a $\Pi^0_2$ class $\mathcal{A}$ such that $\hat{\mu}(\mathcal{A}) \geq r$ according to Definition \ref{dfn:2-POS-half}.
\end{lemma}

\begin{proof}
 Fix $\mathfrak{M} = (M,\mathcal{S}) \models \RCA$ and $T$ and $r$ as in the statement. If $T$ is bounded in height then $[T]$ is empty, its measure is zero, and the proof is trivial, so assume that it is not. Let $\langle T_0, T_1, \ldots\rangle$ be a uniformly $\Delta^0_1$ sequence of trees converging pointwise to $T$, as given by the previous Lemma.  Further assume, by cutting the tops off if necessary, that each $T_s$ is empty above level $s$. Our construction is as follows.

 Whenever $\sigma$ is in some $T_s$, select an unused $j$ and set $\thoo B_{|\sigma|,j} = [\sigma] = \{X \in 2^M  : \sigma$ is an initial segment of $X\}$. Some care is needed in selecting the $j$'s so that the resulting $\Pi^0_2$ class is strictly presented, but this is easily done: for instance, $j = \langle f(\sigma), s \rangle$ will do, where $f(\sigma) = 2^{|\sigma|} + \sum_{k < |\sigma|} \sigma(k)\cdot 2^k$. For all other $i,j$, just set $\thoo B_{i,j}$ to equal some other $\thoo B_{i,j'}$ (which may involve waiting, if no other $\thoo B_{i,j'}$ has yet been defined).

It remains to verify that $\mathcal{A} = \bigcap_i \bigcup_j \thoo B_{i,j}$ equals $[T]$, and has measure $\hat \mu(\mathcal{A}) \geq r$. The former claim is a consequence of the starting assumption that each $T_s$ be downard closed and of height $\leq s$. For the latter claim, it is enough to notice that if the tree's $i$-th level $T\cap 2^i$ has exactly $k$ elements, corresponding to a measure of $k/2^i$, then the union $\bigcup_j \thoo B_{i,j}$ contains at least $k$-many disjoint cylinders each of measure $2^{-i}$, totalling $\geq k/2^i$.
\end{proof}

\begin{corollary}[$\RCA$]\label{cor:POS-implies-WWKL}
$2\hyp\POS(r)$ implies $2\hyp\WWKLax(r)$ for all positive $r\in\QQ$.
\end{corollary}

\begin{proof}
Just apply Lemma \ref{lem:POS-implies-WWKL}.
\end{proof}

To get the other direction, we look at $\Delta^0_2$ trees that behave regularly.

\begin{definition}\label{def:regular-trees}
A \defm{regular $\Delta^0_2$ tree}
  in a model $\str M=(M,\thoo S)\models\RCA$
 is a $\Delta^0_2$ tree $T$ over $\str M$
  such that $(T \cap 2^i: i < n)$ is $\str{M}$-finite for every $n\in M$.
\end{definition}

Such a regularity property does not come for free in a nonstandard world:
 if $\str M\not\models\bd\Sigma^0_2$,
  then it is not hard to produce a $\Delta^0_2$-definable function $F\colon M\to 2$ and $n\in M$
   such that $(F(i):i<n)$ is not $\str M$-finite.
One can avoid this irregularity, provably in $\RCA$,
 by passing on to a $\Delta^0_2$-definable non-decreasing cofinal sequence
  whose elements increase sufficiently rarely.

\begin{lemma}\label{lem:coded-sseq}
In a model $(M,\thoo S)\models\RCA$,
 if $F\colon M\to M$ is total $\Delta^0_2$,
  then there is a $\Delta^0_2$ sequence $(n_i)_{i \in M}$
   such that $i \leq j$ implies $i \leq n_i \leq n_j$ and
    the function $$j \mapsto \seq{F(n_i) : i < j}$$ is total $\Delta^0_2$.
\end{lemma}
\begin{proof}
Work in $(M,\thoo S)$.
Let $F = \lim_s F_0(\cdot,s)$ as in the Limit Lemma~\ref{limit-lemma}.
Define the \defm{modulus function} $m\colon M\to M$ by
 \begin{equation*}
  m(x)=\min\{s>x:\fage ts{F_0(x,s)=F_0(x,t)}\}.
 \end{equation*}
Then $m$ is total and $\Pi^0_1$ (i.e.\ its graph $\{(x,s): m(x)=s\}$ is $\Pi^0_1$),
 as one can easily verify using $\ind\Sigma^0_1$.
If we write $m^k$ to mean $m$ composed with itself $k$ times,
 then the (possibly partial) function $k \mapsto m^k(0)$ is strictly increasing and $\Sigma^0_2$.
The domain of this function is clearly closed under successor.
Therefore,
 since $\ind\Sigma^0_1$ implies $\{\tuple{k,n}:m^k(0)=n\leq i\}$ is coded
   for every $i\in M$,
  we see that $\fa i{\ex k{(m^k(0) > i)}}$.
Define the sequence $(n_i)_{i \in M}$ by
 \begin{equation*}
  n_i=\min\{m^k(0):m^k(0)>i\}.
 \end{equation*}
Notice if $i,j,k,\ell\in M$
  such that $m^k(0)=n_i$ and $m^\ell(0)=n_j$, where $i<j$,
 then $m(n_i)=m^{k+1}(0)\leq m^{\ell+1}(0)=m(n_j)$ and
  so $F(n_i)=F_0(n_i,m(n_j))$.
A moment of thought then reveals $j\mapsto\tuple{F(n_i):i<j}$ is total and $\Delta^0_2$.
\end{proof}

The following proposition is a strengthening of Avigad et al.~\cite[Proposition 3.4]{Avigad.Dean.ea:2012},
 which is in turn a formalization of Kurtz~\cite[p.~21, Lemma 2.2a]{Kurtz:81.thesis},
  which is ultimately just an effective account of the regularity property of the Lebesgue measure on the Borel sets.
The only difference between our proposition and that in \cite{Avigad.Dean.ea:2012} is
 that we require only $\RCA$, rather than $\RCA + \bd \Sigma^0_2$;
  and the only real difference between our proof and that in \cite{Avigad.Dean.ea:2012} is
   that by more carefully defining a certain function $F$ we are able to replace an appeal to $\bd \Sigma^0_2$ with one to the lemma above.

\begin{proposition}[$\RCA$]\label{prp:subtree-Pi2}
If $\thoo A$ is a $\Pi^{0}_2$ class and $\hat{\mu}(\thoo A) \geq r > \delta>0$,
 then there exists a regular $\Delta^0_2$ tree $T$
  such that $[T] \subseteq\thoo A$ and $\mu([T]) \geq r - \delta$.
\end{proposition}

\begin{proof}
Work in a model $(M,\thoo S)\models\RCA$.
Let
$$
  \thoo A=\bigcap_i \bigcup_j \thoo{B}_{i,j}
$$ be a strictly presented $\Pi^0_2$ class,
 where $\thoo{B}_{i,j} = [\sigma_{i,j}]$ and
  the map $(i,j) \mapsto \sigma_{i,j}$ is $\Sigma^0_1$.
Moreover, we may assume that $|\sigma_{i,j}| > i$.
For each $i$, let
\begin{equation*}
 F(i) = \min\Bigl\{ k :
  \text{$\mu\Bigl(\bigcup_{j<\ell}\thoo B_{i',j}-\bigcup_{j<k}\thoo B_{i',j}\Bigr)<\frac\delta{2^{i+1}}$
        for all $\ell>k$ and all $i'\leq i$}
 \Bigr\}.
\end{equation*}
Then $F$ is a $\Delta^0_2$ total function by $\ind\Sigma^0_1$.
Let $(n_i)_{i \in M}$ be as in Lemma~\ref{lem:coded-sseq} for $F$.
Intuitively, we will define $T$
 such that $[T]=\bigcap_i\bigcup_{j<F(n_i)}\thoo B_{i,j}$.
This will ensure $[T]\subseteq\thoo A$ and $\mu([T])\geq r-\delta$
 if $\hat{\mu}(\thoo A)\geq r>\delta>0$ by the definition of~$F$.

Formally, we define a tree $T$ as follows.
For each binary sequence $\tau$ of length $\ell$,
 put $\tau \in T$ if and only if
  for each $i\leq\ell$ there is some $j<F(n_i)$
   such that $\sigma_{i,j}$ is comparable with $\tau$.
It follows straight from the definition that $T$ is a tree.
Moreover $T$ is a regular $\Delta^0_2$ tree
 because $T\cap 2^\ell$ can be computed uniformly
  from $\seq{F(n_i):i\leq\ell}$.
The rest is a simple exercise.
%
\end{proof}

From this lemma we can derive a partial reversal of Corollary \ref{cor:POS-implies-WWKL}.

\begin{proposition}[$\RCA$]\label{prp:WWKL-implies-POS}
If $q$ and $r$ are elements of $\QQ$ and $0 < q < r$
 then $2\hyp\WWKLax(q)$ restricted to regular $\Delta^0_2$ trees
  implies $2\hyp\POS(r)$. \qed
\end{proposition}

To prove a full reversal of Corollary~\ref{cor:POS-implies-WWKL},
  i.e.,~the equivalence of $2\hyp\POS(r)$ and $2\hyp\WWKLax(r)$ over $\RCA$,
 we use the Weak Pigeonhole Principle.
In the next lemma, let us establish our first connection
 between $2\hyp\WWKL(1/2)$ and $\Sigma^0_2\hyp\WPHP$,
  by a proof similar to Avigad et al's~\cite{Avigad.Dean.ea:2012} proof
    that $2\hyp\WWKL$ implies $\bd\Sigma^0_2$,
   or the classical proof that $\WKLax$ implies $\bd\Sigma^0_1$~\cite[Proposition~5]{art:feasthy-analysis}.

\begin{lemma}\label{lem:WWKL-WPH}
$2\hyp\WWKL(1/2)$ restricted to regular $\Delta^0_2$ trees
 proves $\Sigma^0_2\hyp\WPHP$.
\end{lemma}

\begin{proof}
Fix a model $\str{M} = (M,\thoo{S}) \models \RCA + \neg \Sigma^0_2\hyp\WPHP$.
Then there is a $b \in M$ for which $\str M\not\models\Sigma^0_2\colon 2b \rightarrow (2)^1_b$.
By Lemma~\ref{lem:WPH-basics}(\ref{lem:WPH-basics/2}), the set of such $b$'s is closed upwards in $M$;
 so let us choose $b$ to be a power of $2$, say $2b = 2^k$. Identify the binary strings of length $k$ uniquely with the numbers $<2b$, e.g.\ by placing them in alphabetical order.

 Use Lemma \ref{lem:WPH-vs-convergence}(\ref{lem:WPHP-vs-convergence/2}) to fix a $\Pi^0_1$ formula $\phi$ such that $\str M \models (\forall x < 2b)(\exists y) \phi(x,y)$, but for every $\ell \in M$, the set
  $$\{x < 2b : \str M\models(\exists y < \ell)\phi(x,y)\}$$
 has size strictly less than $b$. Define a tree $T$ by
  $$T = 2^{< k} \cup \left\{\sigma \in 2^{\geq k} :  \str M\models(\forall y<|\sigma|)\neg\phi(\sigma \upharpoonright k,y)\right\}.$$
 Then $T$ is regular $\Delta^0_2$ and has measure $\geq 1/2$, but has no infinite paths. So $2\hyp\WWKL(1/2)$ fails.
\end{proof}

Our proof of the equivalence of $2\hyp\WWKLax(1/2)$ and $2\hyp\POS(1/2)$
  over $\RCA$
 invokes a weak form of the Lebesgue Density Theorem,
  which will also be useful in Section \ref{s:fo-thy}
   where we demonstrate the conservativity of
    $2\hyp\WWKL(1/2)$ over $\ind\Sigma_1+\Sigma_2\hyp\WPHP$.
\begin{definition}
If $T$ is a tree and $\sigma$ is a finite binary sequence,
 then let $T_\sigma=\{\tau:\sigma\tau\in T\}$.
\end{definition}
Fix $\str M = (M, \thoo S) \models \RCA$.
If $T$ is a tree and $\sigma$ is an $\str M$-finite binary sequence
 then $T_\sigma$ is also a tree, and is computable in $T$.
Moreover, $T_\sigma$ is a regular $\Delta^0_2$ tree if $T$ is.
If we work in a standard model---meaning $M = \mathbb{N}$---then
 by the Lebesgue Density Theorem every tree $T$ with $\mu([T]) > 0$
  has nodes $\sigma$ such that $\mu([T_\sigma])$ is very close to $1$.
This fails in general for $\Delta^0_2$ trees
  in the absence of $\ind\Sigma^0_2$ (see \cite{Chong.Li.ea:2019}),
 but with $\Sigma^0_2\hyp\WPHP$ it holds partially.

\begin{lemma}[Partial Lebesgue Density]\label{lem:WPH-partial-Density}
The following is provable in $\RCA+\Sigma^0_2\hyp\WPHP$ for all nonzero $n,m\in\IN$.
\begin{quote}
For any regular $\Delta^0_2$ tree $T$ of measure $\geq1/n$,
 there is a string $\sigma$ such that $T_\sigma$ has measure $\geq 1-1/m$.
Furthermore, given any $b$, we can ensure that $|\sigma| \geq b$.
\end{quote}
\end{lemma}

\begin{proof}
Fix a model $(M,\thoo S)\models\RCA+\Sigma^0_2\hyp\WPHP$ to work in.
Let $T$ be a regular $\Delta^0_2$~tree of measure $\geq1/n$.
Since $m$ and $n$ are standard, we may choose the unique $k$ such that
 \begin{equation}\label{eqn:partial-density}
 \frac k{2mn} \leq \mu([T]) < \frac{k+1}{2mn}.
 \end{equation}
(This is the only place where we use standardness.)
Notice that $k \geq 2m$.
Now fix a level $\ell_0$ at which
 \begin{equation*}
 \frac k{2mn} \leq \frac{|T \cap 2^{\ell_0}|}{2^{\ell_0}} < \frac{k+1}{2mn},
 \end{equation*} and
 assume towards a contradiction that $\mu([T_\sigma]) < 1 - 1/m$
  for all $\sigma \in T \cap 2^{\ell_0}$;
 in other words, each $\sigma \in T \cap 2^{\ell_0}$ has a level $\ell$
  at which $|T_\sigma \cap 2^{\ell}|/2^\ell< 1-1/m$.
The function mapping $\sigma$ to the least such $\ell$ is $\Delta^0_{2}$,
 so by $\Sigma^0_2 \hyp \WPHP$ and Lemma~\ref{lem:WPH-vs-convergence}(\ref{lem:WPHP-vs-convergence/4}),
  there is an $\ell_1 \in M$ such that $|T_\sigma \cap 2^{\ell_1}|/2^{\ell_1}< 1-1/m$
   for more than half of these $\sigma$'s.
Hence
 \begin{align*}
  \frac{|T \cap 2^{\ell_0+\ell_1}|}{2^{\ell_0+\ell_1}}
  &<\frac1{2^{\ell_0+\ell_1}}\left(
     \frac{\card{T\cap2^{\ell_0}}}2\times1\times2^{\ell_1}
     +\frac{\card{T\cap2^{\ell_0}}}2\times\left(1-\frac1m\right)\times2^{\ell_1}
    \right)\\
  &=\frac{1+(1-1/m)}2\times\frac{\card{T\cap2^{\ell_0}}}{2^{\ell_0}}\\
  &<\frac{2m -1}{2m}\times\frac{k+1}{2mn},
 \end{align*}
 which is $<k/{2mn}$ since $k \geq 2m$.
But this means $\mu([T]) < k/{2mn}$, contradicting line (\ref{eqn:partial-density}).

As for the `furthermore' part of the lemma,
 simply note that when selecting $\ell_0$,
  we may select it to be larger than any given $b\in M$.
\end{proof}

\begin{theorem}\label{thm:WWKL=POS}
The following statements are equivalent over $\RCA$ for all $r\in(0,1)\cap\QQ$.
\begin{enumerate}
 \item $2\hyp\POS(1/2)$.
 \item $2\hyp\WWKLax(1/2)$.
 \item $2\hyp\WWKLax(1/2)$ restricted to regular $\Delta^0_2$ trees.
 \item $2\hyp\POS(r)$.
 \item $2\hyp\WWKLax(r)$.
 \item $2\hyp\WWKLax(r)$ restricted to regular $\Delta^0_2$ trees.
\end{enumerate}
\end{theorem}

\begin{proof}
The implications $(1)\Rightarrow(2)$ and $(4)\Rightarrow(5)$
 are special cases of Corollary~\ref{cor:POS-implies-WWKL}.
The implications $(2)\Rightarrow(3)$ and $(5)\Rightarrow(6)$ are trivial.
We claim that $2\hyp\WWKL(q)$ restricted to regular $\Delta^0_2$~trees
 implies $2\hyp\POS(r)$ for all $q,r\in(0,1)\cap\QQ$.
This suffices to entail the remaining implications
 $(3)\Rightarrow(4)$ and $(6)\Rightarrow(1)$.

Work over $\RCA$ plus $2\hyp\WWKLax(q)$
 restricted to regular $\Delta^0_2$~trees.
By Lemmas~\ref{lem:WWKL-WPH} and~\ref{lem:WPH-partial-Density},
 we know $2\hyp\WWKL(r/2)$ holds.
So $2\hyp\POS(r)$ follows from Proposition~\ref{prp:WWKL-implies-POS}.
\end{proof}

In view of the equivalences above,
 it does not matter whether we use
  Definition~\ref{dfn:2-WWKL-half} or Definition~\ref{dfn:2-POS-half}
   when we speak of measure of $\Delta^0_2$ trees.
For the sake of consistency, we will adopt the former in what follows.

Theorem~\ref{thm:WWKL=POS} implies the following strengthening
 of a theorem by Avigad et al.~\cite[Proposition 3.6]{Avigad.Dean.ea:2012}.

\begin{corollary}\label{cor:2-WWKL-2-RAN}
$\RCA\proves 2\hyp\WWKLax(1/2)\to2\hyp\RAN$. \qed
\end{corollary}

%% file: fo-theory.tex
\section{Conservativity}\label{s:fo-thy}

In this section,
 we will determine the first-order theory of~$2\hyp\WWKL(1/2)$.
In particular, we will see that it is finitely axiomatizable.

\begin{theorem}\label{thm:FOT_2-WWKL-half}
$2\hyp\WWKL(1/2)$ is $\Pi^1_1$-conservative over $\RCA + \Sigma^0_2\hyp\WPHP$.
Hence $\ind\Sigma_1 + \Sigma_2\hyp\WPHP$ axiomatizes the first-order theory of $\RCA + 2\hyp\WWKL(1/2)$.
\end{theorem}

\begin{proof}
Conservativity will follow from
 the model expansion theorem~\ref{thm:model-expansion} below
  in the usual manner.
The remaining part is provided by Lemma~\ref{lem:WWKL-WPH}.
\end{proof}

The relevant model expansion theorem will occupy us
 for the rest of this section.
\begin{definition}
A model $(M,\thoo S)\models\RCA$ is \emph{principal}
 if $\thoo S = \{ X \subseteq M : X \Tredeq Z\}$ for some $Z \subseteq M$.
If $\str M = (M, \thoo S)$ and $X\subseteq M$,
 then
 \begin{equation*}
  \str M[X] = (M, \{Y \subseteq M : Y \Tredeq Z \oplus X~\mathrm{for~some}~Z \in \thoo S\}).
 \end{equation*}
We read $\str M[X]$ as \emph{$\str M$ expanded by $X$}.
\end{definition}

\begin{theorem}\label{thm:model-expansion}
Let $\str{M}$ be a countable principal model of $\RCA + \Sigma^0_2\hyp\WPHP$.
If $T$ is a regular $\Delta^0_2(\str M)$ tree of measure $\geq 1/2+\epsilon$,
  where $\epsilon\in\QQ$ with $0<\epsilon<1/2$,
 then there exists $X \in [T]$ such that $\str{M}[X]$ is also a model of $\RCA + \Sigma^0_2\hyp\WPHP$.
\end{theorem}


Our proof involves a relativized version of Jockusch--Soare forcing
 with the restriction that conditions must in a sense have large measure.
The relativization will be carried out carefully
 so as to keep down the amount of induction used.
In particular, we will start the forcing construction
 from a specially chosen condition, to be given by Lemma~\ref{lem:obese-tree-GL1}.
Let us start by setting up some notation.
\begin{definition}[$\RCA^*$]
Recall that $\Phi_e^Z(x)$ denotes the $e$-th Turing functional,
  run with oracle~$Z$ on input $x$, and
 $\Phi_{e,s}^Z(x)$ is the same but run for only $s$-many steps.
Let $\Phi_e^Z(\sigma;x)$ abbreviate $\Phi_{e,|\sigma|}^{Z \oplus \sigma}(x)$.
\end{definition}
For example, with these notations,
 the Turing jump of $Z \oplus X$ can be expressed in more than one way:
  $$(Z \oplus X)' = \{e : \Phi_e^{Z \oplus X}(e) \defd\} = \{e : \exists \ell \,\Phi_e^Z(X \restd \ell ; e) \defd\}.$$

Lemma \ref{lem:obese-tree-GL1} below is the same as Proposition~1.3 in Conidis--Slaman~\cite{Conidis.Slaman:2013},
 except that the base theory is weakened from $\RCA + \bd\Sigma^0_2$ to $\RCA$.
For this improvement,
 we carefully replace applications of $\bd\Sigma^0_2$ with those of Lemma~\ref{lem:omega-re-function}.

\begin{lemma}[$\RCA$]\label{lem:omega-re-function}
Suppose $F_0$ is a recursive function with two arguments, and
 there is a recursive function $G$ such that for all $x$,
\[
  |\{s: F_0(x,s) \neq F_0(x,s+1)\}| < G(x).
\]
Then $F(x) = \lim_s F_0(x,s)$ exists for all $x$, and
 the map $n \mapsto \seq{F(x) : x < n}$ is total and $\Delta^0_2$.
\end{lemma}

\begin{proof}
Define $G^*(n) = G(0) + \dots + G(n-1)$.
Then $G^*$ is total $\Delta^0_1$, and
\[
  |\{s : (\exists x < n) F_0(x,s) \neq F_0(x,s+1)\}| < G^*(n)
\] for all $n$.
Since the set displayed above is coded,
 each $n$ corresponds to some maximum value $s_n$ of $s$
  (or, let us say, to zero should the set be empty);
 moreover, the function mapping $n$ to $s_n$ is $\Delta^0_2$.
From this we get a total $\Delta^0_2$ function mapping each $n$ to
 $\seq{F(x) : x < n} = \seq{F_0(x,s_n+1) : x < n}$.
\end{proof}

Now we state and prove Lemma \ref{lem:obese-tree-GL1}.
 The key idea of the proof is borrowed from the classical proof that
  every $2$-random is generalized low.

\begin{lemma}\label{lem:obese-tree-GL1}
Let $\str M=(M,\thoo S)\models\RCA$ and $Z\in\thoo S$.
Then for every positive $\epsilon\in\QQ\evalin{\str M}$
 there are a non-decreasing $\Delta^{0,Z}_2$ function~$H$ and
           a regular $\Delta^{0,Z}_2$ tree~$\tilde T$
            of measure $\geq1-\epsilon$
  such that for every $X\in[\tilde T]$,
   \begin{equation*}
    (Z\oplus X)'=\{e:\Phi_e^Z(X\restd H(e);e)\defd\}
   \end{equation*} in $\str M[X]$.
\end{lemma}

\begin{proof}
Work in~$\str M$.
Fix a rational $\epsilon>0$.
For each pair $(e,\ell)$, define
 \begin{equation*}
  p_{e,\ell} = \frac{|\{\sigma\in2^\ell : \Phi_e^Z(\sigma;e)\undefd\}|}{2^\ell}.
 \end{equation*}
For any fixed $e$, we have
  $1 \geq p_{e,0} \geq \cdots\geq p_{e,\ell} \geq \cdots \geq 0;$
 thus if we define a $Z$-recursive binary function~$F_0$ by $F_0(e,0)=0$, and
  \begin{equation*}
   F_0(e,s+1) =
   \begin{cases}
    s+1,      & \text{if $p_{e,F_0(e,s)}-p_{e,s+1}\geq\epsilon/2^{e+1}$;}\\
    F_0(e,s), & \text{otherwise,}
   \end{cases}
  \end{equation*}
 there are at most $\floor{2^{e+1}/\epsilon}$-many $s$'s
  for which $F_0(e,s) \neq F_0(e,s+1)$.
Hence by Lemma~\ref{lem:omega-re-function} the pointwise limit
  $F(\cdot)=\lim_s F_0(\cdot,s)$ exists, and
 the map $n\mapsto\seq{F(e):e<n}$ is total and $\Delta^{0,Z}_2$.
From this we obtain a non-decreasing total function
 $H\colon n\mapsto\max\{F(e):e\leq n\}$.
As the reader can directly verify,
 the image of~$H$ is unbounded, and
 the map $n\mapsto\seq{H(e):e<n}$ is total $\Delta^{0,Z}_2$.
Define a tree~$\tilde T$ level-by-level as follows:
 all the coded binary strings of length strictly less than~$H(0)$ are in~$\tilde T$, and
 if $H(n)\leq\ell<H(n+1)$,
  then the $\ell$-th level of~$\tilde T$ is
  \begin{equation*}
   \tilde T_\ell = \{\sigma \in 2^\ell : \falt en{(
    \text{either $\Phi^Z_e(\sigma ; e) \undefd$
              or $\Phi^Z_e(\sigma\restd F(e);e) \defd$}
   )}\}.
  \end{equation*}
It is straightforward to check that
 $\tilde T$~is a regular $\Delta^{0,Z}_2$ tree, and
 $X\in[\tilde T]$ implies $\fa e{(\text{$\Phi^{Z\oplus X}_e(e)\undefd$ or $\Phi^Z_e(X\restd H(e);e)\defd$})}$.

As for the measure, consider the complement $\tilde T_\ell\compl$ of $\tilde T_\ell$.
Whenever $H(n)\leq\ell<H(n+1)$, we have $F(e)\leq H(n)\leq\ell$ for all $e<n$, and
 so
 \begin{align*}
  \frac{|\tilde T_\ell\compl|}{2^\ell}
   &= \frac{\left|\{\sigma\in2^\ell : \exlt en{(\text{$\Phi_e^Z(\sigma;e)\defd$ and
                                                      $\Phi_e^Z(\sigma\restd F(e);e)\undefd$})}\}\right|}
       {2^\ell}\\
   &\leq \sum_{e < n} \frac{\left|\{\sigma\in2^\ell:\text{$\Phi_e^Z(\sigma;e)\defd$ and
                                                          $\Phi_e^Z(\sigma\restd F(e);e)\undefd$}\}\right|}
                       {2^\ell}\\
   &= \sum_{e<n} (p_{e,F(e)}-p_{e,\ell})
   \quad<\quad \sum_{e < n} \frac\epsilon{2^{e+1}}
   \quad<\quad \epsilon.
 \end{align*}
 Hence $\mu([\tilde T]) \geq 1-\epsilon$, as desired.
\end{proof}

We follow this up with a quick, technical lemma.

\begin{lemma}\label{lem:obese-tree-ISigma1}
Fix a principal model $\str M=(M, \{Y \subseteq M : Y \leq_T Z\})\models\RCA$.
Let $H$~be a non-decreasing $\Delta_2^{0,Z}$~function, and let $X\subseteq M$
 be a regular set such that
       $(Z\oplus X)'=\{e:\Phi_e^Z(X\restd H(e);e)\defd\}$ in $\str M[X]$.
Then $\str M[X]\models\RCA$ and
 every $\Sigma^{0, Z \oplus X}_2$ set is $\Sigma^{0,Z' \oplus X}_1$ in $\str M[X]$.
\end{lemma}
\begin{proof}
To get $\RCA$, it is enough to show
 that $\str M[X] \models \ind \Sigma^0_1$, and
 for this it suffices to show
  that $(Z \oplus X)'\restd b$ is $\str M$-finite for all $b\in M$.
Since $H$ is non-decreasing, in $\str M[X]$ we have
 \begin{equation*}
  (Z \oplus X)' \restd b
   = \{e <b  : \Phi_e^Z( X \restd H(b) ; e) \defd\}.
 \end{equation*}
This set, when evaluated in~$\str M$, must be $\str M$-finite
 because $X\restd H(b)$ is $\str M$-finite and
         $\str M\models\ind\Sigma^0_0+\exp$.

As $\str M[X]\models\bd\Sigma^0_1+\exp$,
 every $\Sigma^{0,Z \oplus X}_2$ set is $\Sigma^{0,(Z \oplus X)'}_1$ in~$\str M[X]$.
So for the Lemma's second claim, it is enough to show that $(Z \oplus X)' \Tredeq Z' \oplus X$.
Consider the definition of $(Z \oplus X)' \restd b$ given in the line displayed above.
Since $H$ is $\Delta^{0,Z}_2$, we can compute this set from $Z' \oplus X$ by
 first asking for the value of $H(b)$;
 then asking for $X\restd H(b)$ and $Z\restd H(b)$; and
 finally checking whether each $\Phi_e^{Z\restd H(b)}(X\restd H(b);e)$ halts or not.
\end{proof}

%

Now we are ready for the actual construction.\addtocounter{theorem}{-3}
\begin{proof}[Proof of Theorem~\ref{thm:model-expansion}]
Fix a principal countable model
     $\str M=(M,\{Y\subseteq M:Y\Tredeq Z\})\models\RCA+\Sigma^0_2\hyp\WPHP$ and
    a rational $\epsilon\in(0,{1/2})$.
Consider any regular $\Delta^0_2(\str M)$ tree~$T$
 of measure $\geq1/2+\epsilon$.
We will force a branch through~$T$
 while preserving $\RCA+\Sigma^0_2\hyp\WPHP$.

Let $\tilde T$ be the result of applying Lemma~\ref{lem:obese-tree-GL1}
 to $\str M$, $Z$ and~$\epsilon$.
One can readily see that $T^*\defeq T\cap\tilde T$ is
 a regular $\Delta^0_2(\str M)$ tree of measure
  at least $(1-\epsilon)+(1/2+\epsilon)-1=1/2$ by the inclusion--exclusion principle.
Forcing conditions are pairs $(\sigma,S)$ where
 $\sigma\in T^*$ and
 $S$~is a regular $\Delta^0_2(\str M)$ subtree
  of $T^*_\sigma$ with measure $\geq1/2$.
A condition $(\hat\sigma,\hat S)$ \defm{extends} another condition $(\sigma,S)$
 if $\hat\sigma=\sigma\tau$ and $\hat S\subseteq S_\tau$
  for some $\str M$-finite binary string $\tau$.
Every sufficiently generic sequence $(\sigma_i,S_i)_{i\in\IN}$ of conditions
 gives rise to a subset $G\subseteq M$
  whose characteristic function is $\bigcup_{i\in\IN}\sigma_i$.

\begin{claim}
For every $\ell\in M$ and every condition $(\sigma,S)$,
 there is a condition $(\hat\sigma,\hat S)$ extending~$(\sigma,S)$
  in which the length of $\hat\sigma$ is at least~$\ell$.
Hence if $G$ is sufficiently generic then $\str{M}[G] \models \RCA$.
\end{claim}

\begin{poclaim}
The first part is a direct consequence of
 the Partial Lebesgue Density Lemma \ref{lem:WPH-partial-Density}.
Thus every sufficiently generic set $G\in[T^*]\subseteq[\tilde T]$.
So $\str M[G]\models\RCA$ by Lemma~\ref{lem:obese-tree-GL1} and Lemma~\ref{lem:obese-tree-ISigma1}.
\end{poclaim}

In view of Lemma~\ref{lem:WPH-vs-convergence},
           Lemma~\ref{lem:obese-tree-ISigma1}, and
           the claim above,
 it suffices to show that a sufficiently generic extension~$\str M[G]$
  satisfies the following requirements
   for all Turing functionals $\Phi$ and all $b\in M$:
   \begin{equation*}
    \mathcal{R}_{\Phi,b}\colon
    \text{$(\exists a < b) \Phi^{ Z' \oplus G}(a) \undefd$
       or $(\exists \ell) \bigl|\bigl\{a < b:
            \Phi^{Z'}(G\restd \ell; a) \defd
           \bigr\}\bigr| \geq \frac{b}{2}$}.
   \end{equation*}
The next two claims demonstrate how one can satisfy these requirements and
 thus finish the proof.
Fix a Turing functional~$\Phi$ and $b\in M$.
Pick any condition $(\sigma,S)$.
For each $a<b$, let
$$
    S_a = \{\tau \in S: \str M\models\Phi^{Z'}(\sigma\tau; a) \undefd\}.
$$
As $S$ is a regular $\Delta^0_2(\str M)$ tree, so is $S_a$.

\begin{claim}\label{clm:forcing-div}
If $a\in M$ with $a<b$ and $\mu([S_a])\geq 2^{-n}$ for some $n \in \mathbb{N}$,
 then $(\sigma,S)$ has an extension forcing $\Phi^{Z'\oplus X}(a)\undefd$.
\end{claim}

\begin{poclaim}
Immediate from the Partial Lebesgue Density Lemma \ref{lem:WPH-partial-Density}.
\end{poclaim}

\begin{claim}\label{clm:forcing-card}
Suppose that $\mu([S_a])< 2^{-n}$ for all $a<b$ and all $n \in \mathbb{N}$.
Then $(\sigma,S)$ has an extension $(\hat{\sigma},\hat{S})$ such that
$$
    |\{a < b: \str M\models\Phi^{Z'}(\hat{\sigma}; a) \defd\}| \geq \frac{b}{2}.
$$
\end{claim}

\begin{poclaim}
Work in $\str M$.
For each~$\ell$, define
 \begin{equation*}
  q_\ell=\frac{|\{\tau\in S\cap2^\ell:
   \text{$\Phi^{Z'}(\sigma\tau;a)\undefd$ for more than half of all $a<b$}
  \}|}{2^\ell}.
 \end{equation*}
If there is an $\ell^*$ such that $q_{\ell^*}< 1/4$,
 then we can define $(\hat\sigma,\hat S)=(\sigma\tau,S_\tau)$,
  using a $\tau\in S$ of length $\geq\ell^*$ obtained by
   applying the Lebesgue Density Lemma~\ref{lem:WPH-partial-Density} to
    the regular $\Delta_2^{0,Z}$~tree
    \begin{equation*}
     S\cap\bigl(2^{<\ell^*}\cup\{\tau\in2^{\geq\ell^*}:
      \text{$\Phi^{Z'}(\sigma\tau;a)\defd$
            for more than half of all $a<b$}
     \}\bigr),
    \end{equation*}
    which has measure $\geq 1/4$.
So we only need to show that such an $\ell^*$ exists.

Fix any value $n \in \mathbb{N}$.
We know by the premise of this claim that
 for each $a<b$ there is an~$\ell$ such that
  \begin{equation*}
   \frac{|\{\tau\in S\cap2^\ell:\Phi^{Z'}(\sigma\tau;a)\undefd\}|}{2^\ell}
   < 2^{-n};
  \end{equation*}
 moreover, there exists a $\Delta^{0,Z}_2$ function taking $a$ to such an $\ell$.
Note that if $\ell$ satisfies the inequality above
 then so does every $\ell' > \ell$.
Thus we may apply $\Sigma^0_2\hyp\WPHP$
 together with Lemmata~\ref{lem:WPH-basics}(\ref{lem:WPH-basics/1}) and
   \ref{lem:WPH-vs-convergence}(\ref{lem:WPHP-vs-convergence/4})
  to obtain an $\ell^*$ such that
   the inequality displayed above holds with $\ell=\ell^*$
    for more than $(1-2^{-n})b$ many $a<b$.

Define three coded sets $A,B,C$ as follows:
 \begin{align*}
  A&=S \cap 2^{\ell^*},\\
  B&=\{\tau\in A:\text{$\Phi^{Z'}(\sigma\tau; a)\undefd$ 
                       for more than half of all $a<b$}\},\\
  C&=\{(a,\tau)\in[0,b-1]\cross A:\Phi^{Z'}(\sigma\tau; a)\defd\}.
 \end{align*}
On the one hand, by the choice of $\ell^*$,
$$
  |C| > (1-2^{-n})b (|A|-2^{\ell^*-n}).
$$
On the other hand,
 if we count the elements $(a,\tau)\in C$ with $\tau\in B$ and
             those with $\tau\in A-B$ separately,
 then we get
 $$
  |C| < \frac b2|B| + b(|A|-|B|) = b|A| - \frac b2|B|.
 $$
Putting the two inequalities together, we deduce that
 $$
  |B| < |A| 2^{-n+1} + (1-2^{-n})2^{\ell^*-n+1} \leq 2^{\ell^*} (2^{-n+2} - 2^{-2n+1}).
 $$
In particular, when we use the value $n=4$,
 this implies $q_{\ell^*} = \card B/2^{\ell^*} < 1/4$, as required.
\end{poclaim}
%
\end{proof}

%% file: bsigma-wphp-csigma.tex
\section{Strength}\label{s:wph-fragments}

Here we compare $\Sigma_{n+1}\hyp\WPHP$ with $\bd\Sigma_{n+1}$ and
 the cardinality scheme for $\Sigma_{n+1}$ formulas ($\Cd\Sigma_{n+1}$)
  introduced by Seetapun and Slaman~\cite{Seetapun.Slaman:1995.Ramsey}.
For a set of formulas $\Gamma$, \emph{the cardinality scheme for $\Gamma$},
 denoted by $\Cd\Gamma$, asserts that
  every total injection defined by a formula in $\Gamma$ has an unbounded range.
(Note that this is different from what Kaye defines in his paper~\cite[Section~3.1]{art:Th(kappa-like)}.)
In view of Theorem~\ref{thm:B=PHP},
$$
  \ind\Delta_0+\exp \vdash (\bd\Sigma_{n+1} \to \Sigma_{n+1}\hyp\WPHP) \wedge (\Sigma_{n+1}\hyp\WPHP \to \Cd\Sigma_{n+1}).
$$
The main results of this section state that
 these arrows do not reverse for any $n\in\IN$.

We will prove the independence of $\Sigma_{n+1}\hyp\WPHP$ over $\ind\Sigma_n+\exp + \Cd\Sigma_{n+1}$
 using two drastically different model-theoretic constructions.
In the first construction, we cofinally extend
 any countable $\str M \models \ind\Sigma_n+\exp+\neg \Sigma_{n+1}\hyp\WPHP$
  to $\str N \models \ind\Sigma_n+\exp+\neg \Sigma_{n+1}\hyp\WPHP + \Cd\Sigma_{n+1}$
   by modifying a coded ultrapower construction due to Paris \cite{Paris:1981}.
In the second construction,
 we build an end-extension chain $(\str M_i: i \in \mathbb{N})$ of models of~$\PA$
  such that $\bigcup_i \str M_i \models \ind\Sigma_n+\exp+\neg \Sigma_{n+1}\hyp\WPHP + \Cd\Sigma_{n+1}$.
The first construction has better control over which~$r$
  we have $\fa b{(\Sigma_{n+1}\colon rb\to (2)^1_b)}$ in the model,
 while the second construction naturally gives singular-like models.
Both constructions will be revisited in Section~\ref{s:further}.

For the independence of $\bd\Sigma_{n+1}$ over $\ind\Sigma_n + \exp + \Sigma_{n+1}\hyp\WPHP$,
 we modify Paris's coded ultrapower construction again to
  cofinally extend any given countable $\str M \models \ind\Sigma_n + \exp + \neg \bd\Sigma_{n+1}$
   to $\str N \models \ind\Sigma_n + \exp + \Sigma_{n+1}\hyp\WPHP + \neg \bd\Sigma_{n+1}$.

As the reader may have already noticed,
 Paris's coded ultrapower construction is one of the key techniques we will use.
A model of arithmetic $\str N$ is a \emph{cofinal extension} of another model $\str M$,
  denoted $\str M \cfsub\str N$,
 if $M \subseteq N$ and every $b \in N$ is below some $a \in M$.
We also write $\str M \elemsub_{\Sigma_{n}, \cf}\str N$
 if $\str M \subseteq_{\cf}\str N$ and $\str M \elemsub_{\Sigma_{n}}\str N$, etc.
The following theorem, which is extracted from a proof in Paris's paper~\cite[Theorem~11]{Paris:1981},
 summarizes the major features of the construction.
The $n=0$ case is not mentioned there,
 but it can be proved in the same way.

\begin{theorem}[Paris \cite{Paris:1981}]\label{thm:Paris-ultrapower}
Fix $n \in \IN$ and $\str M \models \ind\Sigma_n + \exp$.
Suppose $b\in M$ and $\UF$ is an ultrafilter on the Boolean algebra of $\str M$-finite subsets of $b$.
Let
$$
  \str N = \str M \cap M^b / \UF.
$$
Then
\begin{enumerate}
  \item \L o\'s's Theorem holds for every $\Sigma_n$ or $\Pi_n$ formula~$\theta$,\label{thm:Paris-ultrapower/1}
   i.e.,~whenever $[f_0],[f_1],\ldots,[f_{k-1}] \in N$,
    \begin{equation*}
     \str N \models \theta([f_0],[f_1],\ldots,[f_{k-1}])
    \end{equation*}
    if and only if
    \begin{equation*}
     \{i \in M:\text{$i < b$ and $\str M \models \theta(f_0(i),\ldots,f_{k-1}(i))$}\} \in \UF;
    \end{equation*}
  \item $\str M \elemsub_{\Sigma_{n+1}, \cf}\str N \models \ind\Sigma_n + \exp$.\label{thm:Paris-ultrapower/2} \qed
\end{enumerate}
\end{theorem}

It is not hard to see
  (from Paris's original application or from our applications below)
 that $\Sigma_{n+1}$~elementarity here cannot be `improved' to
  $\Sigma_{n+2}$~elementarity in general.

\input{wphp-csigma.tex}

\input{bsigma-wphp.tex}

%% file: wphp-csigma.tex
\subsection{The cardinality scheme}\label{ss:wphp-csigma}

As alluded to above,
 we present two proofs of the following independence theorem.

\begin{theorem}\label{thm:CSigma-WPHP}
$\ind\Sigma_n + \exp + \Cd\Sigma_{n+1} \nproves \Sigma_{n+1}\hyp\WPHP$
 for any $n\in\IN$.
\end{theorem}

The ultrapower proof, which is the more elementary one here, comes first.

\begin{lemma}\label{lem:CSigma-cfx}
Fix $n\in\IN$.
Let $\str M$ be a countable model of $\ind\Sigma_n+\exp$, and
 $a,b,e\in M$ with $a\geq2$ and $e\not\in\IN$.
Suppose we have a $\Sigma_{n+1}$~injection $F\colon a^e b \to b$ in~$\str M$.
Then there exists a countable $\str N\elemext_{\Sigma_{n+1},\cf}\str M$
  satisfying $\ind\Sigma_n+\exp$
 in which $[0,a]\evalin{\str M} = [0,a]\evalin{\str N}$
  and $F(c)\evalin{\str N}$ is undefined at some $c < a^e b$.
\end{lemma}

\begin{proof}
The extension $\str N$ will be the ultrapower with respect
 to an ultrafilter
  $\UF \subset \mathcal{P}(a^e b)\evalin{\str M} = \str M \cap \mathcal{P}(a^e b)$
   constructed as follows.

First, apply Limit Lemma~\ref{limit-lemma} to obtain
 a $\Sigma_n$ function $F_0$ approximating~$F$ over $\ind\Sigma_n$.
This implies, in particular, that
 whenever $\str M\elemsub_{\Sigma_{n+1}}\str M'\models\ind\Sigma_n$,
  if $F\evalin{\str M'}$ is a total function $[0,a^eb-1]\evalin{\str M'}\to M'$,
   then $F_0\evalin{\str M'}$ is a total function
    $[0,a^eb-1]\evalin{\str M'}\times M'\to M'$ and
    \begin{equation*}
     \str M'\models \falt x{a^e b}{(F(x) = \lim_s F_0(x,s))}.
    \end{equation*}
For each $s\in M$, let
$$
  A_s = \{x < a^e b: \forall t > s (F_0(x,t) = F_0(x,s))\}\evalin{\str M}.
$$
Note that $A_s \subseteq A_t$ whenever $t > s$.
As $\str M \models \ind\Sigma_n+\exp$,
 we know $A_s \in M$ and $\str M \models |A_s| \leq b$.
So $(A_s: s \in M)$ generates a proper ideal $\mathcal{A}$
 on the Boolean algebra $\mathcal{P}(a^e b)\evalin{\str M}$.

Then we construct $\UF$. Let $(h_k: k \in \mathbb{N})$ list all $\str M$-finite $a^e b \to a$.
Let $X_0 = [0, a^e b - 1]\evalin{\str M}$.
If $X_k \in M$ is defined and $\str M\models|X_k| \geq a^{e-k} b$,
 then define $X_{k+1} = X_k \cap h_k^{-1}(i_k) \in M$ where $i_k < a$
  such that $\str M\models|X_{k+1}| \geq a^{e-k-1} b$.
The result is a descending sequence $(X_k: k \in \mathbb{N})$ in $\mathcal{P}(a^e b)\evalin{\str M}$
 such that $h_k(X_{k+1})$ is a singleton and $\str M\models|X_k| > b$ for every $k\in\IN$.
Hence, the filter $\mathcal{F}$ generated by $(X_k: k \in \mathbb{N})$ is disjoint from $\mathcal{A}$.
Let $\UF$ be any ultrafilter on $\mathcal{P}(a^e b)\evalin{\str M}$
 that contains $\mathcal{F}$ and is disjoint from $\mathcal{A}$.

Let
$$
    \str N = \str M \cap M^{a^e b} / \UF.
$$
For each $[h] \in N$ below $a$,
 there exists $i < a$ in $M$ such that $h^{-1}(i) \in \UF$\
  by construction and
  by \L o\'s's Theorem for Paris's construction,
   i.e.,~Theorem~\ref{thm:Paris-ultrapower}\partref{thm:Paris-ultrapower/1}.
This implies, via \L o\'s's Theorem again,
 that $[0,a]\evalin{\str M} = [0,a]\evalin{\str N}$.
Let $c \in N$ represented by the identity function on $a^e b$.
By \L o\'s's Theorem, $c <\evalin{\str N} a^e b$.
We claim that $\lim_s F_0(c,s)$ is undefined in $\str N$.
Suppose not.
Then $\str N \models \forall t > s(F_0(c,t) = F_0(c,s))$ for some $s \in N$.
As $\str M \cfsub\str N$, we may assume this $s\in M$.
By \L o\'s's Theorem,
$$
  A_s = \{x < a^e b: \forall t > s(F_0(x,t) = F_0(x,s))\}\evalin{\str M} \in \UF,
$$
contradicting the condition that $\UF$ is disjoint from $\mathcal A$.
Hence $F(c)\evalin{\str N}$ is undefined.
The remaining properties of $\str N$ follow from
 Paris's Theorem~\ref{thm:Paris-ultrapower}\partref{thm:Paris-ultrapower/2}.
\end{proof}

Repeating the ultrapower construction above
 leads to the model extension theorem below.

\begin{theorem}\label{thm:CSigma-WPHP-cfx}
Fix $n\in\IN$.
Let $\str M$ be a countable model of $\ind\Sigma_n+\exp$ and $a \in M$.
Then there exists $\str N\elemext_{\Sigma_{n+1},\cf}\str M$
  satisfying $\ind\Sigma_n+\exp$
 such that $[0,a]\evalin{\str M} = [0,a]\evalin{\str N}$ and
  $\str N \models \Sigma_{n+1}\colon a^e b \to (2)^1_b$
   for all $e \in N - \mathbb{N}$ and $b \in N$.
\end{theorem}

\begin{proof}
If $\str M$ is standard, then there is nothing to do.
So assume $\str M$ is nonstandard.
By repeated applications of Lemma \ref{lem:CSigma-cfx},
 we can obtain a sequence $(\str N_k: k \in \mathbb{N})$
  such that for every $k\in\IN$,
\begin{enumerate}
 \item $\str M = \str N_0 \elemsub_{\Sigma_{n+1},\cf} \str N_k \elemsub_{\Sigma_{n+1},\cf} \str N_{k+1} \models \ind\Sigma_n+\exp$;
 \item $[0,a]\evalin{\str M} = [0,a]\evalin{\str N_k}$; and
 \item for each $\Sigma_{n+1}$ injection $F\colon a^e b \to b$ in $\str N_k$
   where $b \in N_k$ and $e \in N_k - \IN$,
  there exists $\ell > k$
   such that $F\evalin{\str N_\ell}$ is undefined
    at some $c \in N_\ell$ below $a^e b$.
\end{enumerate}
Then $\str N = \bigcup_{k \in \mathbb{N}}\str N_k$ satisfies the requirements,
 as the reader can readily verify.
\end{proof}

\begin{proof}[First Proof of Theorem \ref{thm:CSigma-WPHP}]
Groszek and Slaman~\cite[Proposition~3.1]{Groszek.Slaman:1994}
  produced a countable model $\str M \models \ind\Sigma_1$ with $a\in M$
 such that some $\Delta_2\evalin{\str M}$ injection $M\to[0,a]$.
Relativization gives a countable model $\str M \models \ind\Sigma_n+\exp$ with $a\in M$
  such that some $\Delta_{n+1}\evalin{\str M}$ injection $M\to[0,a]$.
In particular, there exists a $\Sigma_{n+1}\evalin{\str M}$ injection $F\colon 2a \to a$.
Apply Theorem~\ref{thm:CSigma-WPHP-cfx} to this $\str M$ and this~$a$ to get $\str N$.
Note that $[0,a]\evalin{\str M} = [0,a]\evalin{\str N}$
 implies $[0,2a]\evalin{\str M} = [0,2a]\evalin{\str N}$.
Therefore, by the elementarity between the models,
 $F\evalin{\str N}$ is a $\Sigma_{n+1}$ injection $2a\to a$ in~$\str N$.
So $\str N \models \ind\Sigma_n + \exp + \Cd\Sigma_{n+1} + \neg \Sigma_{n+1}\hyp\WPHP$.
\end{proof}

The second proof of Theorem~\ref{thm:CSigma-WPHP} originates from a construction
 devised by Theodore A. Slaman in around~2011;
 see Haken~\cite[Chapter~3]{phd:haken}.
What allows us to improve on Slaman's construction is the following recent theorem
 from Blanck~\cite[Theorem~5]{unpub:hier-incompl}.
Here $\Pi_n\hyp\Tr$ denotes the set of
 all (standard and nonstandard) $\Pi_n$~sentences
  that are declared true by the usual satisfaction predicate
   for $\Pi_n$~formulas.
An extension of a model of arithmetic is an \defm{end extension}
 if all new elements are above all old elements.
End extensions are indicated by a subscript~$\ee$.

\begin{theorem}[Blanck]\label{thm:flexextn}
Let $n\in\IN$ and $T\supseteq\PA$
 be a recursively axiomatized theory
  in a language extending the language of first-order arithmetic.
Then there exists a $\Sigma_{n+1}$~formula~$\theta(x)$
  such that
  \begin{enumerate}
  \item \begin{math}
    \PA\proves\ex s{\fa x{\bigl(\Ackin xs\nsc\theta(x)\bigr)}}
   \end{math};\label{cond:flexextn/coded}
  \item \begin{math}
    \PA\proves\Con(T+\Pi_n\hyp\Tr)\nsc\fa x{\neg\theta(x)}
   \end{math}; and\label{cond:flexextn/con}
  \item for every $\str M\models T$ and every $s\in M$,
   if $\str M\models\fa x{(\theta(x)\then\Ackin xs)}$,
    then $\str M$~has a $\Sigma_n$-elementary end extension
     $\str K\models T+\fa x{(\theta(x)\nsc\Ackin xs)}$.\label{cond:flexextn/flex}
  \qed
  \end{enumerate}
\end{theorem}

The next theorem illustrates how one can use Blanck's theorem
 to define any specific countable set in a singular-like end extension.
Recall that a linearly ordered structure~$\str M$ is \defm{$\kappa$-like},
  where $\kappa$~is a cardinal,
 if $\str M$~has cardinality~$\kappa$,
  but every proper initial segment of~$\str M$ has cardinality
   strictly less than~$\kappa$.
\begin{theorem}\label{thm:sing-like}
Fix $n\in\IN$, a recursive theory $T_0\supseteq\PA$, and
 a cardinal~$\kappa$ of countable cofinality.
Let $\str M\models T_0$ of cardinality strictly less than~$\kappa$.
For every countable $A\subseteq M$,
 there is a $\kappa$-like $\Sigma_n$-elementary end extension $\str K$ of~$\str M$
  satisfying~$\Sigma_{n+3}\hyp\Th(T_0)$ in which $A\in\Sigma_{n+1}\hyp\Def(\str K)$.
\end{theorem}

\begin{proof}
Without loss of generality, assume $\str M$ is nonstandard.
Fix a strictly increasing sequence of cardinals~$(\kappa_j)_{j\in\IN}$
 whose supremum is~$\kappa$ and
 whose first element~$\kappa_0$ is strictly bigger than the cardinality of~$\str M$.
Use Craig's Trick to find a recursive sequence $(\pi_k(v))_{k\in\IN}$
  of $\Pi_{n+2}$~formulas
 such that $\{\ex v{\pi_k(v)}:k\in\IN\}$ axiomatizes $\Sigma_{n+3}\hyp\Th(T_0)$.
Then use $\ind\Sigma_{n+3}$ to get $c\in M$
 which makes $\str M\models\pi_k((c)_k)$ for all $k\in\IN$.
Let $T=T_0+\{\pi_k((\dname c)_k):k\in\IN\}$,
 where $\dname c$~is a fresh constant symbol.
Notice $(\str M,c)\models T$.
We will build a sequence
 \begin{equation*}
  \str M=\str M_0\elemsub_{\Sigma_n,\ee}\str M_1
                 \elemsub_{\Sigma_n,\ee}\str M_2
                 \elemsub_{\Sigma_n,\ee}\cdots
 \end{equation*}
 such that each $(\str M_{j+1},c)$ is a $\kappa_j$-like model of~$T$.
This automatically ensures
 $\str K=\bigcup\{\str M_j:j\in\IN\}$ is a $\kappa$-like $\Sigma_n$-elementary
  end extension of~$\str M$ satisfying~$\Sigma_{n+3}\hyp\Th(T_0)$.

Let $\theta(x)$ be a $\Sigma_{n+1}$~formula
 satisfying conditions~\partref{cond:flexextn/coded}--\partref{cond:flexextn/flex}
  in Theorem~\ref{thm:flexextn}.
Use condition~\partref{cond:flexextn/coded} there
 to find $b\in M$ such that $\str M\models\fa x{(\theta(x)\then x<b)}$.
Fix an enumeration $(a_j)_{j\in\IN}$ of~$A$.

Now, given any model
 \begin{math}
  (\str M_j,c)\models T+\fage xb{\bigl(
   \theta(x)\nsc\bigvvee_{i<j}x=b+a_i
  \bigr)}
 \end{math} of cardinality less than~$\kappa_j$,
 we can apply Theorem~\ref{thm:flexextn}\partref{cond:flexextn/flex}
  to obtain $\str M_{j+1}\elemext_{\Sigma_n,\ee}\str M_j$
   satisfying
   \begin{math}
    T+\fage xb{\bigl(
     \theta(x)\nsc\bigvvee_{i<j+1}x=b+a_i
    \bigr)}
   \end{math}.
Moreover,
 in view of the L\"owenheim--Skolem Theorem and
            the Mac~Dowell--Specker Theorem~\cite[Theorem~8.6]{Kaye:1991},
  this $\str M_{j+1}$ can be chosen to be $\kappa_j$-like.
This ensures
 \begin{equation*}
  A=\{a\in K:\str K\models\theta(b+a)\}
   \in\Sigma_{n+1}\hyp\Def(\str K)
 \end{equation*} at the end.
\end{proof}

\begin{proof}[Second Proof of Theorem~\ref{thm:CSigma-WPHP}]
Let $\str M$~be a countable nonstandard model of~$\PA$.
Fix any nonstandard $a\in M$ and any bijection $f\colon 2a\to a$.
Apply Theorem~\ref{thm:sing-like} to
 $\kappa=\beth_\omega$
 and
 \begin{math}
  A=\{\tuple{x,f(x)}:x<2a\}
 \end{math} with $T_0=\PA$.
\end{proof}

%% file: bsigma-wphp.tex
\subsection{The collection scheme}\label{ss:bsigma-wphp}

To show this independence,
 we will start with a countable model of $\ind\Sigma_n+\exp+\neg\bd\Sigma_{n+1}$,
  then repeated apply a suitable version of Paris's coded ultrapower construction
   to achieve~$\Sigma_{n+1}\hyp\WPHP$ in a cofinal extension
    while preserving $\ind\Sigma_n+\exp+\neg\bd\Sigma_{n+1}$.

First, let us see how to preserve the failure of $\bd\Sigma_{n+1}$ in an extension.
In view of Slaman~\cite{Slaman:2004.BSigma},
 between models of $\ind\Delta_0+\exp$,
  this is equivalent to preserving some proper $\Delta_{n+1}$-definable cut.

\begin{definition}
If $\str N$ is a linearly ordered structure and $X\subseteq N$,
 then
 \begin{equation*}
  \sup_{\str N} X = \{x\in N: \exin yX{x\leq\evalin{\str N} y}\}
 \end{equation*}
 and
 \begin{equation*}
  \inf_{\str N} X = \{x\in N : \fain yX{x<\evalin{\str N}y}\}.
 \end{equation*}
\end{definition}

\begin{lemma}[Keita Yokoyama]\label{lem:cut-preserving}
Fix $\str M\models\PAminus$ and $n\in\IN$.
If $I$~is a $\Delta_{n+1}$-definable proper cut of~$\str M$, and
  $\str N\elemext_{\Sigma_{n+1}}\str M$ in which $\sup_{\str N}I=\inf_{\str N}(M-I)$,
 then $J\defeq\sup_{\str N}I$ is a $\Delta_{n+1}$-definable proper cut of~$\str N$.
\end{lemma}

\begin{proof}
Suppose $I=\phi(\str M)=\psi(\str M)$,
 where $\phi(v)\in\Sigma_{n+1}$ and $\psi(v)\in\Pi_{n+1}$,
  both of which may involve parameters from~$M$.
Define
 \begin{align*}
  \phi'(v) &\qeq \exge wv{\phi(w)},\quad\text{and}\\
  \psi'(v) &\qeq \fale uv{\psi(u)}.
 \end{align*}
Since $I$ is a cut of~$\str M$, we know $\phi'(\str M)=\psi'(\str M)=I$.
Notice $\phi'(v)$ is $\Sigma_{n+1}$ and
       $\psi'(v)$ is $\Pi_{n+1}$.
So it suffices to show that $J=\phi'(\str N)=\psi'(\str N)$.

The two directions are symmetric.
So we only show one of them here.
Take $c\in N-J$.
Recall $J=\sup_{\str N}I=\inf_{\str N}(M-I)$.
So we get $d\in M-I$ such that $d\leq c$.
Then, since $d\not\in I=\phi'(\str M)=\psi'(\str M)$,
 \begin{equation*}
  \str M\models\fage wd{\neg\phi(w)}\wedge\exle ud{\neg\psi(u)}.
 \end{equation*}
By $\Sigma_{n+1}$~elementarity,
 the same formula is true in~$\str N$.
Thus, as $d\leq c$,
 \begin{equation*}
  \str N\models\fage wc{\neg\phi(w)}\wedge\exle uc{\neg\psi(u)}. \qedhere
 \end{equation*}
\end{proof}

Our ultrapower construction below is designed
 to correct a failure of $\Sigma_{n+1}\hyp\WPHP$.
That it naturally produces an extension
  satisfying the hypotheses of Lemma~\ref{lem:cut-preserving}
 is rather remarkable, at least at first sight.

\begin{lemma}\label{lem:WPHP-cfx}
Fix $n\in\IN$.
Let $\str M$ be a countable model of $\ind\Sigma_n + \exp$ and
 $(I_k: k \in \mathbb{N})$ be a countable family of cuts of $M$.
Suppose $a\in M$ and $F$ is a $\Sigma_{n+1}\evalin{\str M}$ injection $2a \to a$.
Then there exists a countable $\str N\elemext_{\Sigma_{n+1},\cf}\str M$
 satisfying $\ind\Sigma_n+\exp$
  in which $\sup_{\str N}I_k=\inf_{\str N}(M-I_k)$ for all $k\in\IN$, and
   $F(c)\evalin{\str N}$ is undefined for some $c<2a$.
\end{lemma}

\begin{proof}
We will construct an ultrafilter $\UF$ on $\mathcal{P}(2a)\evalin{\str M}$,
  i.e.,~the power set of $2a$ computed in $\str M$,
 so that the coded ultrapower
$$
    \str N \defeq \str M \cap M^{2a} / \UF
$$ has the required properties.

First, apply Limit Lemma~\ref{limit-lemma} to obtain
 a $\Sigma_n$ function $F_0$ approximating~$F$ over $\ind\Sigma_n$
  as in the proof of Lemma~\ref{lem:CSigma-cfx}.
For each $s \in M$, let
$$
    A_s = \{i < 2a: \forall t > s(F_0(i,t) = F_0(i,s))\}\evalin{\str M}.
$$
As in the proof of Lemma \ref{lem:CSigma-cfx},
 each $A_s$ is $\str M$-finite and $\str M\models|A_s|<a$.
Let $\mathcal{A}$ be the ideal generated by $(A_s: s \in M)$
 in $\mathcal{P}(2a)\evalin{\str M}$.
Say an $\str M$-finite set $X$ is \emph{large}
 if and only if $\str M \models \forall s(|X \cap A_s| < |X|/2)$, and
 \emph{very large} if and only if $\str M \models \forall s(|X \cap A_s| < |X|/4)$.

Then we construct a descending sequence $(X_\ell: \ell \in \mathbb{N})$
 of large $\str M$-finite sets
  starting from $X_0 = [0,2a-1]\evalin{\str M}$,
   which is clearly large.
We employ a forcing-style construction
 to ensure,
  for each $I_k$ and each $\str M$-finite $h\colon 2a \to M$,
   the existence of $y\in M$ and $\ell\in\IN$
    such that either
    \begin{itemize}
    \item $y \in I_k$ and $X_\ell \subseteq \{i < 2a: h(i) \leq y\}\evalin{\str M}$; or
    \item $y \not\in I_k$ and $X_\ell \subseteq \{i < 2a: h(i) > y\}\evalin{\str M}$.
    \end{itemize}
The two claims below can be viewed as density properties of
 an appropriate forcing notion.

\begin{claim}\label{clm:WPHP-cfx-density}
Every large $\str M$-finite set has a very large subset.
\end{claim}

\begin{poclaim}
Fix a large $X \in M$ and work in $\str M$. Let
$$
    j = \max\{i < 8: \exists s (|X \cap A_s| \geq i|X|/8)\},
$$
and let $s$ be such that $|X \cap A_s| \geq j|X|/8$.
Let $Y = X - A_s$.
Then, for each $t$,
 the maximality of $j$ implies
 \begin{math}
  \card{X\cap A_t}<(j+1)\card X/8
 \end{math}
 and hence,
  if $t\geq s$, then
  \begin{equation*}
   \card{Y\cap A_t}
   =\card{X\cap A_t}-\card{X\cap A_s}
   <(j+1)\frac{\card X}8-j\frac{\card X}8
   =\frac{\card X}8
   <\frac{\card Y}4
  \end{equation*} by the largeness of $X$.
\end{poclaim}

\begin{claim}\label{clm:WPHP-cfx-density-1}
For each $k \in \mathbb{N}$, each $\str M$-finite $h\colon 2a \to M$ and each very large $\str M$-finite $Y$, there exist a large $\str M$-finite $Z \subseteq Y$ and $y \in M$ such that either $y \in I_k$ and $Z \subseteq \{i < 2a: h(i) \leq y\}\evalin{\str M}$, or $y \not\in I_k$ and $Z \subseteq \{i < 2a: h(i) > y\}\evalin{\str M}$.
\end{claim}

\begin{poclaim}
Work in $\str M$. Since $h \in M$, the range of $h$ is bounded. Define
\begin{equation*}
    z = \min\{x: |\{i \in Y: h(i) \leq x\}| \geq |Y|/2\}.
 \end{equation*}
If $z \in I_k$, then let $y = z$ and $Z = \{i \in Y: h(i) \leq y\}$.
Clearly $|Z| \geq |Y|/2$ in this case.
If $z \not\in I_k$, then let $y = z - 1\in M - I_k$ and $Z = \{i \in Y: h(i) > y\}$.
The minimality of $z$ ensures $|Z| \geq |Y|/2$ in this case too.
As $Y$ is very large,
$$
    |Z \cap A_s| < |Y|/4 \leq |Z|/2
$$
for all $s \in M$. So $Z$ and $y$ are as desired.
\end{poclaim}

By the countability of $M$ and also of $(I_k)$,
 we can inductively apply the above claims to obtain the $X_\ell$'s we want.
As the reader can readily see,
 the largeness of the $X_\ell$'s implies that
  the filter $\mathcal{F}$ generated by $(X_\ell)$ is disjoint from $\mathcal{A}$.
Let $\UF$ be an ultrafilter on $\mathcal{P}(2a)\evalin{\str M}$
 which contains $\mathcal{F}$ and is disjoint from $\mathcal{A}$, and let
$$
  \str N = \str M \cap M^{2a} / \UF.
$$

\begin{claim}\label{clm:WPHP-cfx-cut}
For each $k \in \mathbb{N}$, $\sup_{\str N} I_k = \inf_{\str N} (M - I_k)$.
\end{claim}

\begin{poclaim}
Let $k \in \mathbb{N}$ and $[h] \in N$.
There are $X \in \UF$ and $y \in M$ such that
 either $h(i) \leq y \in I_k$ for all $i \in X$,
     or $h(i) > y > I_k$ for all $i \in X$.
So by \L o\'s's Theorem,
 either $[h] \in \sup_{\str N} I_k$
     or $[h] \not\in \inf_{\str N} (M - I_k)$.
\end{poclaim}

As in the proof of Lemma \ref{lem:CSigma-cfx},
 we know $\str M \elemsub_{\Sigma_{n+1}, \cf}\str N\models\ind\Sigma_n+\exp$ and
  $F\evalin{\str N}$ is undefined at the element
   represented by the identity function on $[0,2a-1]\evalin{\str M}$.
So $\str N$ satisfies all the properties required by the lemma.
\end{proof}

With Lemmata \ref{lem:cut-preserving} and \ref{lem:WPHP-cfx} at hand,
 separating $\Sigma_{n+1}\hyp\WPHP$ and $\bd\Sigma_{n+1}$
  is only a matter of routine iteration.

\begin{theorem}\label{thm:WPHP-BSigma}
Fix $n \in \mathbb{N}$.
Let $\str M$ be a countable model of $\ind\Sigma_n + \exp$ and
 $(I_k: k \in \mathbb{N})$ be a countable family of cuts of $\str M$.
Then there exists $\str N\elemext_{\Sigma_{n+1},\cf}\str M$
  satisfying $\ind\Sigma_n+\exp+\Sigma_{n+1}\hyp\WPHP$
 such that for all $k \in \mathbb{N}$,
  \begin{equation*}
   \sup_{\str N}I_k=\inf_{\str N}(M-I_k).
  \end{equation*}
Hence $\ind\Sigma_n + \exp + \Sigma_{n+1}\hyp\WPHP \nproves \bd\Sigma_{n+1}$.
\end{theorem}

\begin{proof}
By repeated applications of Lemma \ref{lem:WPHP-cfx},
 obtain a sequence $(\str N_\ell: \ell \in \mathbb{N})$
  such that for all $k,\ell\in\IN$,
\begin{enumerate}
 \item $\str M = \str N_0 \elemsub_{\Sigma_{n+1}, \cf} \str N_\ell \elemsub_{\Sigma_{n+1}, \cf}\str N_{\ell+1} \models \ind\Sigma_n + \exp$;
 \item $\sup_{\str N_\ell} I_k = \inf_{\str N_\ell} (M - I_k)$;
 \item for each $\Sigma_{n+1}$ injection $F\colon2a\to a$ in $\str N_\ell$
   where $a\in N_\ell$,
  there exists $m>\ell$
   such that $F\evalin{\str N_m}$ is undefined
    at some $c\in N_m$ below~$2a$.\label{thm:WPHP-BSigma/3}
\end{enumerate}
Let $\str N = \bigcup_{\ell \in \mathbb{N}}\str N_\ell$.
Then $\str M \elemsub_{\Sigma_{n+1},\cf}\str N$ and
 $\sup_{\str N} I_k = \inf_{\str N} (M - I_k)$ for all $k\in\IN$.
Since $(\str N_\ell)$ is a $\Sigma_{n+1}$-elementary chain,
 the union $\str N$ satisfies
  $\bigcap_{\ell\in\IN}\Pi_{n+3}\hyp\Th(\str N_\ell)\supseteq\ind\Sigma_n+\exp$.

Suppose $a \in N$ and
 $F\colon 2a \to a$ is a $\Sigma_{n+1}\evalin{\str N}$ injection.
Pick a large enough $\ell$ such that $N_\ell$ contains
 $a$ and all the parameters in the definition of $F$.
Then $F\evalin{\str N_\ell}$ is a $\Sigma\evalin{\str N_\ell}_{n+1}$ injection
  $[0,2a-1]\evalin{\str N_\ell}\to[0,a-1]\evalin{\str N_\ell}$
 since $\str N_\ell \elemsub_{\Sigma_{n+1}}\str N$.
So (\ref{thm:WPHP-BSigma/3}) gives $m>\ell$ and $c\in N_m$ below $2a$
 such that $F\evalin{\str N_m}$ is undefined at $c$.
As $\str N_m \elemsub_{\Sigma_{n+1}}\str N$,
 $F\evalin{\str N}$ is undefined at $c$ as well,
  contradicting the assumption on $F$.
This shows that $\str N \models \Sigma_{n+1}\hyp\WPHP$.

If the $\str M$ above does not satisfy $\bd\Sigma_{n+1}$,
 then we can choose $I_0$ to be a proper $\Delta_{n+1}\evalin{\str M}$ cut,
  which ensures $\str N\not\models\bd\Sigma_{n+1}$
   in view of Lemma~\ref{lem:cut-preserving}.
The last part of the theorem follows.
\end{proof}

\begin{remark}
It is apparent that the assumption $\exp$ can be weakened
 in the cofinal extension constructions in this section:
 we only need a theory in which
  we can count the elements of $\Delta_0$-definable sets somehow.
For example, Theorem~\ref{thm:Paris-ultrapower} remains true
 even without~$\exp$ if we replace
  `$\str M$-finite' by `bounded $\Delta_0(\Sigma_n)$-definable' everywhere.
\end{remark}

%% file: further.tex
\section{More about pigeonhole principles}\label{s:further}
As we saw in Sections~\ref{s:wwkl} and~\ref{s:fo-thy},
 the principle $\Sigma_n\hyp\WPHP$ arises naturally
  when one studies the Weak Weak K\"onig Lemma.
Clearly one can obtain a hierarchy of weaker pigeonhole principles
 by similarly changing the domains of the functions involved:
\begin{equation*}
 \begin{tabular}{lclll}
 $\Sigma_n\hyp\PHP(\num+1,\num)$
 &$\defeq$&$\forall a$
  &$\Sigma_n\colon a+1\to (2)^1_a$;\\
 $\Sigma_n\hyp\PHP(2\num,\num)$
 &$\defeq$&$\forall a\geq1$
  &$\Sigma_n\colon 2a\to(2)^1_a$;\\
 $\Sigma_n\hyp\PHP(\num^2,\num)$
 &$\defeq$&$\forall a\geq2$
  &$\Sigma_n\colon a^2\to(2)^1_a$;\\
 $\Sigma_n\hyp\PHP(2^\num,\num)$
 &$\defeq$&$\forall a$
  &$\Sigma_n\colon 2^a\to(2)^1_a$;\\
 \quad\vdots
 &\vdots&
  &\enspace\vdots\\
 $\Sigma_n\hyp\PHP(H(\num),\num)$
 &$\defeq$&$\forall a$
  &$\Sigma_n\colon H(a)\to(2)^1_a$;\\
 \quad\vdots
 &\vdots&
  &\enspace\vdots\\
 $\Sigma_n\hyp\PHP({<}\infty,\num)$
 &$\defeq$&$\forall a\ \exists b$
  &$\Sigma_n\colon b\to(2)^1_a$;\\
 $\Sigma_n\hyp\PHP(\infty,\num)$
 &$\defeq$&$\forall a$
  &$\Sigma_n\colon\infty\to(2)^1_a$.
 \end{tabular}
\end{equation*}
Here $\Sigma_n\hyp\PHP(\num+1,\num)$ is simply
  the usual $\Sigma_n\hyp\PHP$;
 the principle $\Sigma_n\hyp\PHP(2\num,\num)$
  is what we have called $\Sigma_n\hyp\WPHP$; and
 $\Sigma_n\hyp\PHP(\infty,\num)=\Cd\Sigma_n$.
Kaye~\cite[Section~3.2]{art:Th(kappa-like)}
 refers to $\{\Sigma_k\hyp\PHP({<}\infty,\num):k\in\IN\}$
  as a \emph{generalized pigeonhole principle}.

We saw several separation results for this hierarchy
 over $\ind\Delta_0+\exp$ for \emph{positive} $n\in\IN$.
On the one hand, Theorem~\ref{thm:WPHP-BSigma}
 separates $\Sigma_n\hyp\PHP(\num+1,\num)$
  from $\Sigma_n\hyp\PHP(2\num,\num)$.
On the other hand, both constructions in Section~\ref{ss:wphp-csigma}
 can separate $\Sigma_n\hyp\PHP(2\num,\num)$,
  $\Sigma_n\hyp\PHP(\num^2,\num)$, \dots\
  from $\Sigma_n\hyp\PHP({<}\infty,\num)$.
In fact, one can squeeze a little more out of the second construction.
The following improves Theorem~13 in Haken~\cite{phd:haken}.

\begin{theorem}\label{thm:loc-glob}
Let $n\in\IN$.
For any set of $\Sigma_{n+3}$~sentences~$S$
   that is consistent with~$\PA$ and
  any $S$-provably total unary function~$H$
   with a $\Sigma_{n+1}$-definable graph,
 \begin{equation*}
  S+\{\Sigma_k\hyp\PHP({<}\infty,\num):k\in\IN\}
  \nproves\Sigma_{n+1}\hyp\PHP(H(\num),\num).
 \end{equation*}
\end{theorem}

\begin{proof}
Using a universal $\Sigma_{n+1}$~predicate,
 one can finitely axiomatize $\Sigma_{n+1}\hyp\PHP(H(\num),\num)$
  over $\ind\Delta_0+\exp+S$, and
 $\ind\Delta_0+\exp$ itself is well known
  to be finitely axiomatizable~\cite[\S6]{incoll:fragPA+MRDP}.
So we may assume $S$~is finite without loss of generality.
Then run our second proof of Theorem~\ref{thm:CSigma-WPHP},
 changing $T_0$ to $\PA+S$ and $2a$ to~$H(a)$.
\end{proof}

Clearly one can strengthen $\{\Sigma_k\hyp\PHP({<}\infty,\num):k\in\IN\}$
  in Theorem~\ref{thm:loc-glob}
 to any theory satisfied in all $\beth_\omega$-like models
  of $\ind\Delta_0+\exp$.
As shown by Kaye~\cite[Theorem~3.20]{art:Th(kappa-like)},
 such a theory cannot be too strong,
  in the sense that it is always weaker than
  \begin{equation*}
   \IB+\exp\defeq\ind\Delta_0+\exp+\{\ind\Sigma_k\then\bd\Sigma_k:k\in\IN\},
  \end{equation*}
  which is partially conservative over all the usual fragments
   of Peano arithmetic~\cite[Theorem~4.1]{art:Th(kappa-like)}.
In fact, from our proof of Theorem~\ref{thm:loc-glob},
 one sees this implication is strict.
The strictness of this implication,
 which answers a question in Kaye~\cite[Problem~4.3]{art:Th(kappa-like)},
  was first shown by Theodore A. Slaman in around 2011 using a similar method;
   see Haken~\cite[Section~3.3]{phd:haken}.

When $n=0$, the situation is somewhat different:
 as shown by Paris--Wilkie--Woods~\cite[Corollary~2]{art:PWW}
         and Thapen~\cite[Lemma~2.1]{art:model-wphp},
  there is a way to construct
   a counterexample to $\Sigma_0\hyp\PHP(2\num,\num)$ from
   a counterexample to $\Sigma_0\hyp\PHP({<}\infty,\num)$
    in $\ind\Delta_0+\Omega_1$,
 where $\Omega_1$ denotes an axiom asserting
  the totality of $x\mapsto x^{\log x}$ over~$\ind\Delta_0$.
This construction does not work at higher levels of the arithmetic hierarchy
 because apparently one cannot iterate a $\Sigma_n$-definable function
  without increasing the complexity of the defining formula
   when $n\geq1$ and $\bd\Sigma_n$ is absent.
Using a diagonal argument,
 Paris--Wilkie--Woods~\cite[Theorem~1]{art:PWW} showed
  $\ind\Delta_0+\Omega_1\proves\Sigma_0\hyp\PHP(\num^2,\num)$.
So $\ind\Delta_0+\Omega_1\proves\Sigma_0\hyp\PHP(2\num,\num)$ too.
The question whether $\ind\Delta_0+\Omega_1\proves\Sigma_0\hyp\PHP(\num+1,\num)$,
  first raised by Macintyre,
 is a fundamental open question in bounded arithmetic~\cite[Problem~B(c)]{incoll:openprob}.

As observed by Dimitracopoulos and Paris~\cite[Remarks on page~79]{Dimitracopoulos.Paris:1986},
 there is some connection between the $\Sigma_0$ and the $\Sigma_1$~level:
 one can deduce from the Paris--Wilkie--Woods theorem
    in the previous paragraph
   that $\bd\Sigma_1+\Omega_1\proves\Sigma_1\hyp\PHP(2\num,\num)$,
  but the question whether
    $\bd\Sigma_1+\Omega_1\proves\Sigma_1\hyp\PHP(\num+1,\num)$
   is open because it is equivalent to Macintyre's question.

Although $\Sigma_0\hyp\PHP(2\num,\num)$ is known to be strictly weaker
  than $\Sigma_0\hyp\PHP(\num+1,\num)$ in the relativized setting~\cite{art:KPW/exp-lb-php,art:PBI/exp-lb-php},
 we do not yet have an unrelativized separation to date.
In this context, the coded ultrapower construction that
 we used to prove our unrelativized separation
   at higher levels of the arithmetic hierarchy
   (i.e.,~Theorem~\ref{thm:WPHP-BSigma})
  may provide useful information.

Our coded ultrapower constructions in Section~\ref{s:wph-fragments}
 is of independent model-theoretic interest.
Surprisingly little is known about non-elementary cofinal extensions
 of models of arithmetic.
For instance, all such constructions known so far
 make a new collection axiom true in the extension.
Our construction, on the contrary,
 can preserve all failures of collection at the appropriate level.

\begin{question}
Given $n\in\IN$, can one find a model of~$\bd\Sigma_{n+1}$
 with a cofinal extension satisfying $\ind\Delta_0$
  but not~$\bd\Sigma_{n+1}$?
\end{question}

In some sense, one can use
 Lemma~\ref{lem:WPH-basics}\partref{lem:WPH-basics/2} and
 Theorem~\ref{thm:CSigma-WPHP-cfx}
  to characterize $\Sigma_{n+1}\hyp\PHP(2\num,\num)$.

\begin{proposition}
Let $n\in\IN$ and
 $H$~be a provably total unary function in $\ind\Sigma_n+\exp$
  with a $\Sigma_{n+1}$-definable graph.
If
 \begin{math}
  \ind\Sigma_n+\exp+\Sigma_{n+1}\hyp\PHP({<}\infty,\num)
 \end{math}
 proves
 \begin{equation*}
  \fa x{H(x)>x}\quad\text{and}\quad
  \fa r{\fa w{\exge xw{H(x)\geq rx}}},
 \end{equation*}
 then it cannot prove
 \begin{equation*}
  \fa a{\bigl(
   \exge xa{(\Sigma_{n+1}\colon H(x)\to(2)^1_x)}
   \then(\Sigma_{n+1}\colon H(a)\to(2)^1_a)
  \bigr)}.
 \end{equation*}
\end{proposition}

\begin{proof}
Use Theorem~\ref{thm:loc-glob} to find a countable
 $\str M\models\ind\Sigma_n+\exp+\neg\Sigma_{n+1}\hyp\PHP(H(\num),\num)$.
Let $a\in M$ such that $\str M\not\models\Sigma_{n+1}\colon H(a)\to(2)^1_a$.
Fix $e\in M-\IN$.
Set $r=\max\{a,H(a)\}$.
Apply Theorem~\ref{thm:CSigma-WPHP-cfx} to find
 $\str N\elemext_{\Sigma_{n+1}}\str M$ satisfying
 \begin{math}
  \ind\Sigma_n+\exp+\fa b{(\Sigma_{n+1}\colon r^eb\to(2)^1_b)}
 \end{math} such that $[0,r]\evalin{\str M}=[0,r]\evalin{\str N}$.
Notice $\str N\not\models\Sigma_{n+1}\colon H(a)\to(2)^1_a$ as a result.
Hence if the first conjunct in the hypothesis of the proposition is true,
  but the conclusion is not,
 then
 \begin{math}
  \str N\models\fage xa{\neg(\Sigma_{n+1}\colon H(x)\to(2)^1_x)}
 \end{math}, and
 so
 \begin{math}
  \str N\models\fage xa{H(x)<H(a)^ex}
 \end{math}.
\end{proof}

Let us modify our second proof of Theorem~\ref{thm:CSigma-WPHP}
 to show a similar characterization for $\Sigma_{n+1}\hyp\PHP(\num^2,\num)$
  in terms of what we call \emph{$\Sigma_{n+1}$-cardinalities of numbers}.

\begin{definition}
Let $n\in\IN$.
If $\str M\models\ind\Delta_0$ and $a\in M$, then
 \begin{equation*}
  \Sigma_n\hyp\Card\evalin{\str M}(a)=\{ b\in M :
   \text{in $\str M$ some $\Sigma_n$-definable injection $b\to a$}
  \}.
 \end{equation*}
\end{definition}

Clearly, the $\Sigma_n\hyp\Card$ of a number is closed downwards and
 always contains the number itself.
In view of the usual set-theoretic convention,
 it is probably more appropriate to define $\Sigma_n\hyp\Card\evalin{\str M}(a)$ to be
 \begin{equation*}
  \{ b\in M :
   \text{in $\str M$ some $\Sigma_n$-definable injection $b+1\to a$}
  \}.
 \end{equation*}
We choose to adopt a slightly different definition
 because (1)~it actually does not make any difference
   in the cases we are interested in, and
  (2)~it makes the next proof neater.

\begin{proposition}\label{prop:in=cl}
Fix $n\in\IN$.
Let $\str M\models\ind\Delta_0$ and $a\in M$.
\begin{enumerate}
\item $a+1\in\Sigma_n\hyp\Card\evalin{\str M}(a)$ if and only if
 $\Sigma_n\hyp\Card\evalin{\str M}(a)$ is closed under $x\mapsto x+1$.
\item $2a\in\Sigma_n\hyp\Card\evalin{\str M}(a)$ if and only if
 $\Sigma_n\hyp\Card\evalin{\str M}(a)$ is closed under $x\mapsto2x$.
\item $a^2\in\Sigma_n\hyp\Card\evalin{\str M}(a)$ if and only if
 $\Sigma_n\hyp\Card\evalin{\str M}(a)$ is closed under $x\mapsto x^2$.
\end{enumerate}
\end{proposition}

\begin{proof}
The right-to-left directions are obvious.
So let us concentrate on the left-to-right directions.
Fix $b\in\Sigma_n\hyp\Card\evalin{\str M}(a)$ and
 a $\Sigma_n$-definable injection $F\colon b\to a$.
\begin{enumerate}
\item Define $F_1\colon b+1\to a+1$ by setting, for each $x<b+1$,
 \begin{equation*}
  F_1(x)=\begin{cases}
   F(x)+1, &\text{if $x<b$;}\\
   0,      &\text{if $x=b$.}
  \end{cases}
 \end{equation*}

\item Define $F_2\colon2b\to2a$ by setting, for each $i<2$ and $v<b$,
 \begin{equation*}
  F_2(ib+v)=ia+F(v).
 \end{equation*}

\item Define $F_3\colon b^2\to a^2$ by setting, for all $u,v<b$,
 \begin{equation*}
  F_3(ub+v)=F(u)\times a+F(v).
 \end{equation*}
\end{enumerate}
Composing $F_j$ with a witness to the left-hand-side condition
 gives the injection we want.
\end{proof}

Although one may not expect that this list of equivalences goes on forever,
 one may expect at least an analogous equivalence for~$x\mapsto 2^x$.
Nevertheless, this extrapolated equivalence is not true,
 as one can deduce from the following theorem
  by Paris and Mills~\cite[Theorem~2]{art:ParisMills}.

\begin{theorem}[Paris--Mills]\label{thm:PM}
Let $\str M_0$ be a countable model of~$\PA$ and
 $I$~be a cut of~$\str M_0$ closed under multiplication.
Then $\str M_0$~has an elementary extension~$\str M$
 in which $\sup_{\str M}I=I$ and
  every interval $[0,b]\evalin{\str M}$ where $b\in M-I$ is uncountable. \qed
\end{theorem}

Although Corollaries~\ref{cor:am>card} and~\ref{cor:mult-char}
  are formulated in terms of~$\PA$,
 it is not hard to see that they remain true
  when $\PA$~is replaced by any recursively axiomatized consistent
   extension of~$\PA$.

\begin{corollary}\label{cor:am>card}
Fix $n\in\IN$ and a countable $\str M_0\models\PA$.
Let $I$~be a cut of~$\str M_0$ closed under multiplication and $a\in I-\IN$.
Then $\str M_0$~has a $\Sigma_{n+1}$-elementary extension
  $\str K\models\Sigma_{n+3}\hyp\Th(\PA)$
 in which
  \begin{equation*}
   \Sigma_{n+1}\hyp\Card\evalin{\str K}(a)
   =\Sigma_{n+2}\hyp\Card\evalin{\str K}(a)
   =\dots=I.
  \end{equation*}
\end{corollary}

\begin{proof}
Let $\str M$~be an extension of $\str M_0$ given by Theorem~\ref{thm:PM}.
By the L\"owenheim--Skolem Theorem,
 this~$\str M$ can be chosen to have cardinality~$\aleph_1$.
Fix any bijection $f\colon I\to a$.
Apply Theorem~\ref{thm:sing-like} to
 $\kappa=\beth_\omega$
 and
 \begin{math}
  A=\{\tuple{x,f(x)}:x\in I\}
 \end{math} with $T_0=\PA$.
\end{proof}

\begin{corollary}\label{cor:mult-char}
Let $n\in\IN$ and $H$~be a provably total unary function in~$\PA$
 with a $\Sigma_{n+1}$-definable graph.
If $\PA\proves\fa w{\exge xw{H(x)\geq x^k}}$ for all $k\in\IN$,
 then there exist $\str K\models\Sigma_{n+3}\hyp\Th(\PA)$ and $a\in K$
  such that $H(a)\in\Sigma_{n+1}\hyp\Card\evalin{\str K}(a)$
   but $\Sigma_{n+1}\hyp\Card\evalin{\str K}(a)$ is not closed under~$H$.
\end{corollary}

\begin{proof}
Take any nonstandard element~$a$ in a countable model $\str M_0\models\PA$.
By the hypothesis, we know
 $\str M_0\models\exge x{\max\{a,H(a)\}}{H(x)\geq x^k}$ for all $k\in\IN$.
So overspill gives $b\geq\max\{a,H(a)\}$ in~$\str M_0$
 such that $H(b)\geq b^k$ for all $k\in\IN$.
Then apply Corollary~\ref{cor:am>card} to
 $I=\sup_{\str M_0}\{b^k:k\in\IN\}$ to obtain the model~$\str K$ we want.
\end{proof}

Let us conclude with two general questions
 on the strength of weak pigeonhole principles.
Recall that the question
 whether $\ind\Delta_0+\neg\exp+\neg\bd\Sigma_1$ is consistent
  is widely open~\cite[Question~29]{incoll:openprob}.

\begin{question}
Does $\ind\Delta_0+\exp+\Sigma_{n+1}\hyp\PHP(2\num,\num)$
 prove $\ind\Sigma_n$ or $\bd\Sigma_n$ for any $n\geq1$?
\end{question}

\begin{question}
Is $\ind\Delta_0+\neg\exp+\neg\Sigma_1\hyp\PHP(\infty,\num)$ consistent?
\end{question}